\documentclass[psamsfonts]{amsart}

\usepackage{amssymb,amsfonts}
\usepackage[all,arc]{xy}
\usepackage{enumerate}
\usepackage{mathrsfs}
\usepackage{amsmath}
\usepackage{mathtools}
\usepackage[left=3.5cm, right=3.5cm, top=4cm]{geometry}
\usepackage{tikz-cd}
\usepackage{cite}
\usepackage{circuitikz}
\usepackage{comment}
\newtheorem{thm}{Theorem}[section]
\newtheorem{cor}[thm]{Corollary}
\newtheorem{prop}[thm]{Proposition}
\newtheorem{lem}[thm]{Lemma}

\newtheorem{hyp}{Assumption}
\theoremstyle{definition}
\newtheorem{defn}[thm]{Definition}
\newtheorem{defns}[thm]{Definitions}
\newtheorem{con}[thm]{Construction}

\newtheorem{notn}[thm]{Notation}

\newtheorem{innercustomthm}{Theorem}[section]
\newenvironment{customthm}[1]
  {\renewcommand\theinnercustomthm{#1}\innercustomthm}
  {\endinnercustomthm}

\theoremstyle{remark}
\newtheorem{rem}[thm]{Remark}

\newcommand{\mc}[1]{\mathcal{#1}}
\newcommand{\mb}[1]{\mathbb{#1}}
\newcommand{\codim}{\mathrm{codim}}
\newcommand{\proj}[1]{\mathbb{P}#1}
\newcommand{\mbar}{\overline{\mathcal{M}}}

\newcommand{\pone}{\proj{}^1}

\makeatletter
\let\c@equation\c@thm
\makeatother
\numberwithin{equation}{section}

\title{Rationally Simply Connected Hypersurfaces in Orthogonal Grassmannians}

\author{Srijan Ghosh}

\begin{document}
\setlength{\parindent}{0pt}
\begin{abstract}
	In this paper, we study the moduli space of rational curves in a general low degree hypersurface in the Orthogonal Grassmannian $OG(k,n+1)$ of $k$-dimensional isotropic subspaces of an $n+1$-dimensional vector space equipped with a symmetric, non-degenerate, bilinear form. We prove rationally simply connectedness for such a general hypersurface of degree $d$ where $d$ satisfies $n+1-8k-4\ge (3k-1)d^2-d$.
\end{abstract}
\maketitle
\section{Introduction}

Let $X$ be a non-singular, projective, complex variety. Rational curves in $X$ form an important part of studying the structure of $X$. $X$ is said to be rationally connected, if for two general points on $X$, there is a rational curve passing through them. This is similar to the notion of path connectedness in topology. The following theorem is an interesting consequence of rational connectedness:
\begin{thm}\cite[Theorem 1.1]{ghs}
	Let $f:X\rightarrow B$ be a dominant, proper morphism of complex varieties with $B$ a smooth curve. If the geometric generic fiber of $f$ is rationally connected, then $f$ has a section.
\end{thm}

A similar question can be asked about a dominant, proper morphism $f:X\rightarrow S$ from a complex variety to a smooth surface. The notion of rationally simply connectedness provides a criterion for the existence of a generic section in this context. Informally, $X$ being rationally simply connectedness means that the space of rational curves connecting two general points in $X$ if rationally connected for large enough degree. Additionally, we need the existence of a {\it very twisting surface} in $X$. For the complete definition, see Section \ref{sectiondefnrss}. As discussed above, here is one of the most important consequences of rationally simply connectedness:

\begin{thm}\cite[Corollary 13.2]{sectionsoversurfaces}
	Let $S$ be a smooth, irreducible, projective surface over an algebraically closed field of characteristics $0$. Let $f:X\rightarrow S$ be a flat proper morphism. Assume that the geometric generic fiber of $f$ is {\it rationally simply connected}, then there exists a rational section of $f$.
\end{thm}

Additionally, here is a recent result about approximating holomorphic maps to rationally simply connected varieties.

\begin{thm}\cite[Theorem 1.2]{tightapproximation}
	Let $B$ and $X$ be smooth projective algebraic varieties over $\mb{C}$. Assume that $X$ is rationally simply connected and $B$ is a curve. Let $U\subset B(\mb{C})$ be an open subset in the analytic topology. Let $h:U\rightarrow X$ be a holomorphic map. Then there are algebraic maps $B\rightarrow X$ that approximate $h$.
\end{thm}

Here is a partial list of already known families of varieties that are rationally simply connected:
\begin{enumerate}
	\item Projective space and other $G/P$ where $G$ is a semisimple algebraic group and $P$ is a maximal parabolic subgroup \cite{sectionsoversurfaces}.
	\item General low degree complete intersections in projective space \cite{nk1006} and \cite{deland}.
	\item General low degree hypersurfaces in Grassmanians \cite{findley}.
\end{enumerate}
I exapand on these results with the following list:
\begin{customthm}{\ref{mainthm}}
	Let $V$ be an $n+1$ dimensional vector space with a non-denerate, symmetric, bilinear form $Q$, let $OG(k,V)$ denote the space parameterizing $k$-dimensional isotropic subspaces of $V$. Then a general degree $d$ hypersurface  $X\subset OG(k,V)$ satisfying Assumption \ref{rsshyp}, is rationally simply connected.

\end{customthm}

We also get the following result about lines on $X$:
\begin{customthm}{\ref{irredlines}}
	Let $X\subset OG(k,V)$ be a general hypersurface of degree $d$ satisfying Assumption \ref{hyp1}. Then the evaluation map from the space of $1$-pointed lines on $X$:
	$$ev:\mbar_{0,1}(X,1)\rightarrow X$$
	has irreducible fibers of the expected dimension $n-k-d-2$ everywhere except on a closed subset of dimension at most $\frac{k(k-1)}{2}$ where the fibers have dimension $n-k-d-1$. In particular, $\mbar_{0,1}(X,1)$ is irreducible of dimension $\dim X + n-k-d-2$.
\end{customthm}

The proof of Theorem \ref{mainthm} uses existence of twisting and very twisting surfaces on a general hypersurface $X$ under Assumption \ref{rsshyp}. We prove rational connectedness of the space of degree $k+1$ curves connecting two general points of $X$ by studying the space of $(k+1)$-chains of rational lines on $X$ and then deforming. The statement for higher degrees follow using an induction argument involving the twisting surfaces.\\

The paper is structured as follows: Section \ref{sectionallaboutog} discusses the space of lines on $OG(k,n+1)$, Section \ref{sectionlinesinhypersurfaces} and \ref{sectionlinesareirred} deals with the irreducibility of the space of lines on $X$. The notion of stable maps and rationally simply connectedness is introduced in Section \ref{sectiondefnstablemaps} and \ref{sectiondefnrss}. Section \ref{sectiontwistable} covers the existence of twisting and very twisting surfaces. Sections \ref{chains} and \ref{sectiondeformfromchains} discusses rationally connectedness of chains of lines and degree $(k+1)$ rational curves.


\section{Orthogonal Grassmannian}\label{sectionallaboutog}

For the rest of the paper, we will assume the ground field is $\mb{C}$. Let $V$ be an $n+1$-dimensional vector space with a non-degenerate symmetric bilinear form $Q$ on it. Let $k$ be a fixed integer. We say a $k \le \frac{n+1}{2}$-dimensional subspace $W\subset V$ is isotropic if $Q(w_1,w_2)=0$ for all vectors $w_1$ and $w_2$ in $W$.\\
Let $Gr(k,n+1)$ denote the Grassmannian of $k$-dimensional subspaces in $V$. Let $OG(k,n+1)\subset Gr(k,n+1)$ be the locus of $k$-dimensional subspaces isotropic with respect to $Q$. If $\mc{S}_k$ is the tautological vector bundle on $Gr(k,n+1)$, $Q$ corresponds to a section of the rank $\frac{k(k+1)}{2}$ vector bundle $Sym^2(\mc{S}_k^*)$. $OG(k,n+1)$ is exactly the vanishing locus of this section. Using local charts on the grassmannian, we see that this locus has the expected co-dimension. Hence the co-dimension of $OG(k,n+1)$ in $Gr(k,n+1)$ is $\frac{k(k+1)}{2}$ and 
$$dim(OG(k,n+1))=\frac{k(2n-3k+1)}{2}$$
\begin{rem}
A priori the above discussion depends on the form $Q$. However, since all non-degenerate symmetric bilinear forms are equivalent (up to a change of basis) $OG(k,n+1)$ is well-defined up to isomorphism.
\end{rem}
Hereon, we will denote $OG(k,n+1)$ by $OG$ (keeping $k$ and $n$ fixed). We will also use $\mc{S}_k$ for the tautological rank $k$ vector subbundle of $V\otimes \mc{O}_{Gr(k,n+1)}$ on $Gr(k,n+1)$ pulled back to $OG$.
\begin{lem}
	Assuming $n\ge 2k+1$, $OG$ is a smooth, irreducible, projective, Fano variety with $Pic(OG)\cong\mathbb{Z}$. A generator for $Pic(OG)$ is given by the ample line bundle $\mc{O}_{OG}(1)=\det(\mc{S}_k^*)\cong\mathcal{O}_{Gr(k,n+1)}(1)|_{OG}$.
\end{lem}
\begin{lem}
	With respect to the ample bundle described above, $deg(K_{OG})=-(n-k)$.
\end{lem}
\begin{proof}
	On $Gr(k,n+1)$, $OG$ is gives as the zero set of a regular section of $Sym^2(\mc{S}_k^*)$. Hence we have a short exact sequence:
	$$0\rightarrow T_{OG} \rightarrow T_{Gr(k,n+1)}\arrowvert_{OG} \rightarrow Sym^2(\mc{S}_k^*) \rightarrow 0$$
	Therefore $det(T_{OG})=det(T_{Gr(k,n+1)})\arrowvert_{OG} \otimes det(Sym^2(\mc{S}_k^*))^*$. We know that $det(T_{Gr(k,n+1)})$ has degree $n+1$. Since $det(\mc{S}^*_k)$ has degree 1 and $\mc{S}^*_k$ has rank $k$,  $det(Sym^2(\mc{S}_k^*))$ has degree $k+1$. This proves the lemma.
\end{proof}
\subsection{Lines in Orthogonal grassmannians}\label{desclines}
Lines will refer to rational curves that have degree $1$ with respect to $\mc{O}_{OG}(1)$. Let $F_{0,1}(OG)$ denote the space parameterizing lines in $OG$ with a marked point. There is a natural evaluation map $ev: F_{0,1}(OG) \rightarrow OG$. The fiber of this map over a point $p$ is the space of lines passing through $p$, denoted hereon by $F_p(OG)$. We will give an explicit description of these two spaces.
\begin{prop}[]
	Let $p\in OG$ be a point corresponding to an isotropic subspace $W\subset V$. Then there is an isomorphism $F_p(OG)\cong \proj{}^{k-1}\times H$ where $H$ is a non-singular quadric hypersurface in $\proj{}^{n-2k}$.
\end{prop}
\begin{proof}
Note that a line in $OG$ is given by subspaces $W_{k-1} \subset W_{k+1}$ of dimensions $(k-1)$ and $(k+1)$ such that every $k$-dimensional subspace of $W_{k+1}$ is isotropic. This in particular implies that $W_{k+1}$ is isotropic as well. A marked point on this line is given by a subspace $W_k$ of dimension $k$ such that $W_{k-1} \subset W_k \subset W_{k+1}$. Hence $F_{0,1}(OG)$ can be described as the closed locus inside $OG(k-1,n+1)\times OG(k,n+1)\times OG(k+1,n+1)$ where the respective tautological bundles satisfy $\mc{S}_{k-1}\subset \mc{S}_k \subset \mc{S}_{k+1}$. The map $ev$ is the projection onto the second factor.\\
For a fixed $k$-dimensional isotropic subspace $W\subset V$, the choice of a $(k-1)$-dimensional subspace contained in $W$ is given by $\proj{W^*}$. A $(k+1)$-dimensional isotropic subspace $W_{k+1}$ containing $W$ defines a point $[v_{k+1}]$ in $\mathbb{P}(W^{\perp}/W)$. Moreover, the original bilinear form $Q$ induces a non-degenerate, symmetric bilinear form $\bar{Q}$ on $W^{\perp}/W$. For $W_{k+1}$ to be isotropic, we must have $\bar{Q}(v_{k+1},v_{k+1})=0$. This gives us a quadratic hypersurface $H_W$ in $\mathbb{P}(W^{\perp}/W)$ (Note that this has dimension $n-2k-1$). Therefore we have, $F_p(OG)= \mathbb{P}(W^*)\times H_W \cong \mathbb{P}^{k-1}\times H$ where $H$ is a degree $2$, non-singular hypersurface in $\mathbb{P}^{n-2k}$.
\end{proof}

\section{Lines in Hypersurfaces}\label{sectionlinesinhypersurfaces}

We'll state the following assumption for later use:
\begin{hyp}\label{hyp1}
Let $X$ be a degree $d$ hypersurface in $OG=OG(k,n+1)$ where $n \ge 6k$ and $d \le n - 2k - 2$
\end{hyp}
\begin{lem}\label{zerosetofvectorbundle}
Let $X$ be a hypersurface of degree $d$ in $OG$ and fix a point $p\in X$. The locus of lines through $p$ in $X$, given by, $F_p(X)\subset F_p(OG)\cong \proj{}^{k-1}\times H$ is the zero locus of a section of a vector bundle $F$ with a filtration:
	$$F=F_1\supset F_2 \supset \ldots \supset F_{d+1}=0$$
	where $F_i/F_{i+1}\cong \mc{O}(i,i)$.
\end{lem}
\begin{proof}
	From \ref{desclines}, there are three tautological vector bundles on $F_{0,1}(OG)$, namely $\mc{S}_{k-1} \subset \mc{S}_k \subset \mc{S}_{k+1}$. Let $E=\mc{S}_{k+1}/\mc{S}_{k-1}$ and $L=\mc{S}_k/\mc{S}_{k-1}$. Let $i:L \hookrightarrow E$ be the injection of vector bundles. By construction, $\pi:\mb{P}(E)\rightarrow F_{0,1}(OG)$ is the universal family of lines and $i$ induces the section $t:F_{0,1}(OG) \rightarrow \mb{P}(E)$ corresponding to the marked point. Let $h:\mb{P}(E) \rightarrow OG$ denote the morphism mapping each line to its image in $OG$. Let $X$ be a degree $d$ hypersurface in $OG$  defined by $f_X \in \Gamma(OG,\mc{O}(d))$. Let $f^{\prime}_X \in \Gamma(\mb{P}(E),h^*\mc{O}_{OG}(d))$ and $\tilde{f}_X \in \Gamma(F_{0,1}(OG),\pi_*h^*\mc{O}_OG(d))$ be the pull-back and push-forward sections. The vanishing locus of $\tilde{f}_X$ is the set of points $[(l,p)]\in F_{0,1}(OG)$ such that the zero locus of $f^{\prime}_X$ contains the entire fiber $\pi^{-1}([(l,p)])$. That is, the pointed line corresponding to $[(l,p)]$ lies in $X$. Therefore we can identify $F_{0,1}(X)\subset F_{0,1}(OG)$ with the vanishing locus of $\tilde{f}_X$.\\\\
	Since $h^*\mc{O}_{OG}(d)$ is a line bundle with positive degree $d$ on the fibers of $\pi$, the higher direct images vanish and $\pi_*h^*\mc{O}_{OG}(d)$ is a locally free sheaf. Since $\mb{P}(E)$ is a $\mb{P}^1$ bundle over $F_{0,1}(OG)$, we have:
	$$h^*\mc{O}_{OG}(d)=\mc{O}_{\mb{P}(E)}(d)\otimes \pi^*M \implies \pi_*h^*\mc{O}_{OG}(d)=\pi_*\mc{O}_{\mb{P}(E)}(d)\otimes M=Sym^d(E^*)\otimes M$$
	where $M$ is a line bundle on $F_{0,1}(OG)$. Consider the composition $h\circ t: F_{0,1}(OG) \rightarrow OG$. This is the same as the map $ev:  F_{0,1}(OG) \rightarrow OG$ induced by $\mc{S}_k$. Therefore 
	$$t^*(h^*\mc{O}_{OG}(d))=det(\mc{S}_k^*)^d\implies t^*\mc{O}_{\mb{P}(E)}(d)\otimes M=det(\mc{S}_k^*)^d$$
	Since $t$ is given by the inclusion $i:L\hookrightarrow E$, $t^*\mc{O}_{\mb{P}(E)}(d)=(L^*)^d$. Hence:
	$$M=det(\mc{S}_k^*)^d\otimes L^d \implies M=det(\mc{S}_{k-1}^*)^d$$
	Let $N_1=\mc{S}_k/\mc{S}_{k-1}$ and $N_2=\mc{S}_{k+1}/\mc{S}_{k}$. Then we have a short exact sequence of locally free sheaves: 
	$0\rightarrow N_2^* \rightarrow E^* \rightarrow N_1^* \rightarrow 0$.
	Taking the $d$-fold symmetric product we get a filtration $$Sym^d(E^*)=E_0\supset E_1 \supset E_2 \ldots \supset E_{d+1}$$ such that: $$E_i/E_{i+1}=(N_2^*)^i\otimes (N_1^*)^{(d-i)}$$
	Restricting to $F_p(OG)$,  $\mc{S}_k|_{F_p(OG)}$ is a trivial vector bundle. Using the identification $F_p(OG)\cong \mb{P}^{k-1}\times H$, we have $N_1^*=det(S_{k-1})=\mc{O}(-1,0), N_2^*=det(S_{k+1}^*)=\mc{O}(0,1), M=\mc{O}(d,0)$. Restricting the above filtration and tensoring with $M$, we have: $F_p(X) \subset F_p(OG)$ is given by a section $\tilde{f}$ of a vector bundle $F$ with a filtration 
	$$F=F_0\supset F_1 \supset \ldots \supset F_{d+1}=0$$ such that: $F_i/F_{i+1}\cong \mc{O}(i,i)$. In particular the first factor is the structure sheaf $\mc{O}(0,0)$ and $\tilde{f}$ restricts to the zero section of $\mc{O}(0,0)$ if and only if $p\in X$. Therefore, assuming $p\in X$, we may replace $F$ by $F_1$ to finish the proof.
\end{proof} 
Note that $F_p(X)$ has expected co-dimension $d$ in $F_p(OG)$. Therefore the expected dimension of $F_p(X)$ is $n-k-d-2$.
\begin{notn}
	For $X$ a hypersurface of degree $d$ in $OG$, the expected dimension of $F_p(X)$ is $n-k-d-2$ and will be denoted $E(X,p)$
\end{notn}
\begin{rem}
	Note that in the above filtration, the relevant $Ext$ groups vanish and hence we can chose a splitting:
	$$F \cong \mc{O}(1,1)\oplus \mc{O}(2,2) \oplus \ldots \oplus \mc{O}(d,d)$$
\end{rem}

Denote by $\mc{X}$ the universal family of degree $d$ hypersurfaces in $OG$. Let $V_d$ denote the vector space of degree $d$ hypersurfaces and let $\pi_1:\mc{X} \rightarrow \mb{P}(V_d)$ and $\pi_2:\mc{X}\rightarrow OG$ be the projection maps. Now, let us consider the relative space of rational lines in hypersurfaces. Define:
$$F_{0,1}\mathcal{(X)}:=\{(p,l,X) | p\in l \subset X\text{ where $p$ is a point, $l$ is a line and $X$ is a hypersurface}\}$$

Inside $\mc{X}$ define the locus $\mc{Y}_0$ to be the locus of points $(p,X)$ such that in $F_p(OG)$, the section $\tilde{f}_X$ of $F$ defined by $X$ restricted to $F/F_2$ is zero and the dimension of $F_p(X)$ is at least $2$ more than $E(X,p)$.\\
Also, define $\mc{Y}_1$ to be the locus of points $(p,X)$ such that $\tilde{f}_X$ restricted to $F/F_2$ is non-zero and the dimension of $F_p(X)$ is at least $1$ more than $E(X,p)$.\\
Define $\mc{U}=\mc{X}\setminus (\mc{Y}_0\cup \mc{Y}_1)$ to be the good locus. Explicitly, this is the locus of points $(p,X)$ where, either $F_p(X)$ has dimension $E(X,p)$ or $F_p(X)$ is exactly one dimension higher $E(X,p)$ and this defect is caused by the Jacobian of the defining equation of $X$ vanishing along $F_p(OG)$. The main result of this section is the following:

\begin{prop}
	\label{mainprop}
	For a general degree $d$ hypersurface $X$, the pair $(p,X)$ lies in the good locus $\mc{U}$ for all points $p$ in $X$.
\end{prop}

This is equivalent to showing that the projection map $\pi_1$ restricted to $\mc{Y}_0\cup \mc{Y}_1$ is not surjective onto $\mb{P}(V_d)$. For a point $p \in OG$, let $\mc{Y}_{0,p}=\mc{Y}_0\cap \pi_2^{-1}(p)$ and $\mc{Y}_{1,p}=\mc{Y}_1\cap \pi_2^{-1}(p)$.
\begin{lem}\label{equivalentdesc}
	To prove Proposition \ref{mainprop}, it is enough to show that $\codim_{\pi_2^{-1}(p)}\mc{Y}_{i,p} \ge \dim OG$ for $i=0,1$.
\end{lem}
\begin{proof}
	$$\codim_{\pi_2^{-1}(p)}\mc{Y}_{i,p} \ge \dim OG \implies \dim \mc{Y}_{i,p} \le \dim \pi_2^{-1}(p) - \dim OG$$\\
	Note that $\dim \pi_2^{-1}(p) = \dim \mb{P}(V_d) - 1$, thus:
	$$\dim \mc{Y}_{i} \le \dim \mb{P}(V_d) - 1 - \dim OG + \dim OG=\dim \mb{P}(V_d) - 1$$
	Thus $\mc{Y}_i$ can not surject onto $\mb{P}{V_d}$.

\end{proof}
The following lemma on the rank of restriction maps will be useful later on:
\begin{lem}\cite[Lemma ~2.5]{findley}\label{restrictionrank}
	Let $X \subset \proj{}^a\times \proj{}^b$ be a subvariety. Denote by $r_{X,i}$, the rank of the restriction map $H^0(\proj{}^a\times \proj{}^b,\mc{O}(i,i))\rightarrow H^0(X,\mc{O}(i,i)|_{X})$. For a pure variety $X$ of co-dimension $l$, there is a subvariety $W\subset \proj{}^a\times \proj{}^b$ such that:
	\begin{enumerate}
		\item $W$ has co-dimension $l$ in $\proj{}^a\times \proj{}^b$.
		\item  $W$ is a product of linear subvarieites of $\proj{}^a$ and $\proj{}^b$.
		\item $r_{X,i} \ge r_{W,i}$.
	\end{enumerate}
\end{lem}

\begin{rem}
	In the above lemma $W \cong \proj{}^{a-a_1}\times \proj{}^{b-b_1}$ where $a_1+b_1=l$. Thus, the restriction map is actually surjective and $r_{W,i}=\binom{a-a_1+i}{i}\binom{b-b_1+i}{i}=\binom{a-a_1+i}{i}\binom{b+a_1-l+i}{i}$
\end{rem}
Proposition \ref{mainprop} follows from Lemma \ref{equivalentdesc} and the following:
\begin{prop}
	\label{codimprop}
	For $p$ a point in $OG$, and $V_{d,p}$, the space of degree $d$ hypersurfaces through $p$, we have: $\codim_{\proj{V_{d,p}}}\mc{Y}_{i,p} \ge \dim OG$ for $i=0,1$.
\end{prop}
\begin{proof}
	Replace $\mc{Y}_{i,p}$ with the corresponding locus $\mc{Y}^a_{i,p}$  in $V_{d,p}$. It is enough to show $\codim_{V_{d,p}}\mc{Y}^a_{i,p} \ge \dim OG$.\\
	Recall that $F_p(OG)$ is the space of lines through a point $p$ in $OG$ and we can identify: $$F_p(OG) \cong \proj{}^{k-1}\times H \subset \proj{}^{k-1}\times \proj{}^{n-2k}$$
	For a degree $d$ hypersurface $X$ containing $p$, $F_p(X)$ is given by the zero set of the corresponding section $f_X$ of $F$. Denote by $Z_{i,X}$ the zero set of the restriction of $f_X$ to $F/F_{i+1}$. Let $W_{i} \subset V_{d,p}$ denote the locus of hypersurfaces $X$ such that $Z_{i,X}$ has co-dimension less than $i$ in $F_p(OG)$. Denote by $W_i^{b} \subset (W_i)^c$ the locus where $Z_{i,X}$ has co-dimension $i$ and $f_X$ restricted to $F_{i+1}/F_{i+2}|_{Z_{i,X}}\cong \mc{O}(i+1,i+1)|_{Z_{i,X}}$ is a zero divisor. Therefore, we have:
	$$W_{i+1}=W_i \cup W_i^{b}$$

	And: $\mc{Y}^a_{1,p} = (\bigcup\limits_{i=2}^d W_i) \cap W_1^c = (\bigcup\limits_{i=1}^{d-1} W_i^b) \cap W_1^c$. Therefore, it is enough to show that $\codim_{W_i^c}W_i^b \ge \dim(OG)$ for $1\le i\le d- 1$. Consider the vector spaces $A_i=\Gamma(F_p(OG),F/F_{i+1})$. The following restriction maps are surjective:

	$$V_{d,p} \xrightarrow[]{\phi} A_{i+1} \xrightarrow[]{\psi} A_i$$

	It is enough to show that the fibers of the restricted map $\psi|_{\phi(W_i^b)}:\phi(W_i^b) \rightarrow \psi(\phi(W_i^c )) $ has codimension at least $\dim OG$ in fibers of $\psi$. For $h \in \psi(\phi(W_i^c ))$ let $Z_h$ denote the zero locus of $h$ inside $F_p(OG)$. There is a map $\psi^{-1}(h)\rightarrow \Gamma(Z_h,\mc{O}(i+1,i+1)|_{Z_{h}})$ that can be identified with the restriction map $\rho_{Z_h}:\Gamma(F_p(OG),\mc{O}(i+1,i+1))\rightarrow \Gamma(Z_h,\mc{O}(i+1,i+1)|_{Z_{h}})$. We can identify $\psi|_{\phi(W_i^b)}^{-1}(h)$ with the locus in $\psi^{-1}(h)$  whose restriction to $Z_h$ vanishes on an irreducible component of $Z_h$. By restricting to the irreducible component, we may assume $Z_h$ is irreducible. Thus the co-dimension of $\psi|_{\phi(W_i^b)}^{-1}(h)$ in $\psi^{-1}(h)$ is at least the rank of $\rho_{Z_h}$. Note that $Z_h$ has co-dimension $i+1$ in $\proj{}^{k-1}\times \proj{}^{n-2k}$. Furthermore, the restriction maps $\Gamma(\proj{}^{k-1}\times \proj{}^{n-2k},\mc{O}(i+1,i+1))\rightarrow \Gamma(F_p(OG),\mc{O}(i+1,i+1))$ are surjective for all $i$. Thus from Lemma \ref{restrictionrank}, rank of $\rho_{Z_h}$ is at least as large as $\binom{k-a_1+i}{i+1}\binom{n-2k+a_1}{i+1}$ for some $0\le a_1 \le i+1$. Under Assumption \ref{hyp1} and since $i+1 \ge 2$, $\binom{k-a_1+i}{i+1}\binom{n-2k+a_1}{i+1} \ge \binom{n-2k}{2} \ge \frac{k(2n-3k+1)}{2}=\dim OG$. Thus we conclude  $\codim_{\proj{V_{d,p}}}\mc{Y}_{1,p} \ge \dim OG$.\\

	The proof for $\mc{Y}_{0,p}$ is similar. Note that $W_1 \subset V_{d,p}$ is the locus of hypersurfaces $X$ such that $f_X$ restricted to $F/F_2\cong \mc{O}(1,1)$ vanishes. Let $Z_{i,X}$ be the same as above and consider, for $i\ge 2$, $W^{\prime}_i \subset W_1$ the locus of hypersurfaces $X$ such that $f_X$ restricted to $F/F_{i+1}$ has codimension less than $i-1$.  Define $W_i^{b\prime}$ similarly as above to be the locus where $Z_{i,X}$ has co-dimension $i-1$ and $f_X$ restricted to $\mc{O}(i+1,i+1)|_{Z_{i,X}}$ is a zero-divisor. Hence: $\mc{Y}_{0,p}^a=\bigcup\limits_{i=1}^{d-1} W^{b\prime}_i$ and it is enough to show that $\codim_{(W_i^{\prime})^c} W^{b\prime}_i \ge \dim OG$. Let $A_i^{\prime}$ denote $\Gamma(F_p(OG),F_2/F_{i+1})$. Then the following restriction maps are surjective:
	$$W_1 \xrightarrow[]{\phi^{\prime}} A^{\prime}_{i+1} \xrightarrow[]{\psi^{\prime}} A^{\prime}_i$$
	By a similar argument as above, it is enough to check that for $h\in \psi^{\prime}(\phi^{\prime}((W_i^{\prime})^c ))$, 
	$$\codim_{\psi^{\prime-1}(h)} \psi^{\prime}|_{\phi^{\prime}(W_i^{b\prime})}^{-1}(h) \ge \dim OG$$
	Let $Z_h$ be the zero-locus of $h$, we may assume $Z_h$ is irreducible. Thus the above co-dimension is at least the rank of the restriction map $\rho_{Z_h}:\Gamma(F_p(OG),\mc{O}(i+1,i+1)) \rightarrow \Gamma(Z_h, \mc{O}(i+1,i+1))$. In this case, $Z_h$ has co-dimension $i$ in $\proj{}^{k-1}\times \proj{}^{n-2k}$. Using Lemma \ref{restrictionrank}, rank of $\rho_{Z_h}$ is at least $\binom{k-a_1+i}{i+1}\binom{n-2k+a_1+1}{i+1}$ for some $0\le a_1 \le i$. As before, under Assumption \ref{hyp1}, this is bigger than $\dim OG$. Hence we have $\codim_{\proj{V_{d,p}}}\mc{Y}_{0,p} \ge \dim OG$.
\end{proof}


\section{Irreducibility of the Space of Lines}\label{sectionlinesareirred}
\begin{thm}\label{irredlines}
	Let $X$ be a general hypersurface in $OG$ satisfying Assumption \ref{hyp1}. Then $F_{0,1}(X)$ is irreducible of the expected dimension $E(X,p)+\dim X$. A general fiber of $ev:F_{0,1}(X)\rightarrow X$ is irreducible and non-singular of the expected dimension $E(X,p)$.
\end{thm}
First we need the following lemma:
\begin{lem}\label{jumpdimension}
	For a general degree $d$ hypersurface $X$ in OG satisfying Assumption \ref{hyp1}, the locus $Z\subset X$ where the map $ev:F_{0,1}(X)\rightarrow X$ has fiber dimension higher than $E(X,p)$, has dimension at most $\frac{k(k-1)}{2} - 1$.
\end{lem}
\begin{proof}
	From Proposition \ref{mainprop}, we know that for a general $X$, the fiber dimension of $ev$ is at most $E(X,p)+1$ and if $F_p(X)$ is exactly $(E(X,p)+1)$-dimensional, the Jacobian of the defining equation of $X$ vanishes along $F_p(OG)$. Let $\mc{X}$ denote the universal family of degree $d$ hypersurfaces in OG with the projection maps denoted $\pi_1: \mc{X} \rightarrow \proj{V_d}$ and $\pi_2:\mc{X} \rightarrow OG$. Let $\mc{Z}\subset \mc{X}$ be the locus of pairs $(p,X)$ such that the Jacobian of $X$ vanishes along $F_p(OG)$. Let $\phi:\mc{Z} \rightarrow OG$ denote the restriction of $\pi_2$. Using the action of the orthogonal group, it is easy to see that the fibers of $\phi$ are all isomorphic.\\
	Recall that for a fixed point $p$ in $OG$, $F_p(OG)\cong \proj{}^{k-1}\times H$ where $H$ is a non-singular quadric hypersurface in $\proj{}^{n-2k}$. Since the Jacobian of a hypersurface gives a section of the line bundle $\mc{O}(1,1)$ on $F_p(OG)$, $\phi^{-1}(p)$ has dimension equal to $\dim (\proj{V_d}) - 1 - k(n-2k+1)$. Thus $\dim \mc{Z} = \dim (\proj{V_d}) - 1 - k(n-2k+1)+\dim OG = \dim (\proj{V_d}) - 1 + \frac{k(k-1)}{2}$ and, for a general degree $d$ hypersurface $X$, the locus $Z$ has dimension at most $\frac{k(k-1)}{2} - 1$.
\end{proof}
\begin{proof}[Proof of Theorem \ref{irredlines}]
	Let $ev: F_{0,1}(X)\rightarrow X$ denote the evaluation map sending a pointed line to the corresponding point in $X$. From \cite[II.3.11]{kol}, for a general $p\in X$, $ev^{-1}(p)$ is non-singular. From Lemma \ref{zerosetofvectorbundle}, $ev^{-1}(p)$ is the zero locus of a section of an ample vector bundle on $F_{p}(OG)$ with positive expected dimension. Therefore $ev^{-1}(p)$ is connected (see \cite[Theorem ~7.2.1]{lazii}) and hence irreducible for a general $p\in X$. Let $M\subset F_{0,1}(X)$ be an irreducible component dominating $X$. Suppose $N\subset F_{0,1}(X)$ is another irreducible component. From the above discussion, $ev|_N:N\rightarrow X$ is not dominant. Since every irreducible component of $F_{0,1}(X)$ has dimension at least $\dim(X) + E(X,p)$, the fibers of $ev|_N$ must have dimension at least $E(X,p)+1$. From Proposition \ref{mainprop}, Lemma \ref{jumpdimension} and generality of $X$, the fibers of $ev|_N$ has dimension exactly $E(X,p)+1$ and the image $ev(N)$ has dimension at most $\frac{k(k-1)}{2}-1$. Thus $N$ has dimension at most $\frac{k(k-1)}{2}+E(X,p)$ which is strictly less than $\dim(X) + E(X,p)$. Thus we have a contradiction and $F_{0,1}(X)$ is irreducible. Since the fiber dimension $ev^{-1}(p)$ for a general $p\in X$ is $E(X,p)$, $F_{0,1}(X)$ has dimension equal to the expected dimension $dim(X)+E(X,p)$.
\end{proof}


\section{Moduli Space of Stable Maps}\label{sectiondefnstablemaps}
Let $X$ be a smooth projective variety. A stable map with $n$ marked points to $X$ is the datum: $(C,p_1,\ldots,p_n,h)$ where:
\begin{enumerate}
	\item $C$ is an at-worst-nodal curve.
	\item $p_1,\ldots,p_n \in C$ are $n$ points lying in the non-singular locus of $C$.
	\item $h:C\rightarrow X$ a morphism such that the group $Aut(h,p_1,\ldots,p_n)$ of automorphisms of $h$ fixing the points $p_1,\ldots,p_n$ is finite.
	\item[$(3)^{\prime}$] The marked points and the points over nodes in the normalization of $C$ are called special points. For every irreducible component $C^{\prime}$ of $C$ contracted by $h$: If $C^{\prime}$ has genus $0$, there are at least $3$ special points on it and if $C^{\prime}$ has genus $1$, it has at least $1$ special point on it.
\end{enumerate}
The last two conditions are equivalent.\\
Similarly, a family of stable maps with marked points to $X$ is the datum $(\pi:\mc{C}\rightarrow S,\sigma_1,\ldots,\sigma_n,h)$ of a family of at-worst-nodal curves $\pi:\mc{C}\rightarrow S$; $n$-sections $\sigma_1,\ldots,\sigma_n$ of $\pi$; $h:\mc{C}\rightarrow X$ a morphism such that for each $s\in S$, $(C_s,\sigma_1(s),\ldots,\sigma_n(s),h|_{C_s})$ is a stable map with marked points to $X$.\\
Let $\beta \in H_2(X,\mb{Z})$ be a curve class. A stable map $(C,p_1,\ldots,p_n,h)$ has class $\beta$ if $h_*[C]=\beta\in H_2(X,\mb{Z})$. 
\begin{defn}
	Define by $\mbar_{g,n}(X,\beta)$, the Kontsevich moduli space of stable maps with $n$-marked points from genus $g$ curves to $X$ with class $\beta$.
\end{defn}
\begin{thm}\cite[~1.3.1]{kontsevich}
	$\mbar_{g,n}(X,\beta)$ is a proper Deligne-Mumford stack with projective coarse moduli space.
\end{thm}
We will be only interested in the $g=0$ case. Inside $\mbar_{0,n}(X,\beta)$, the locus where the curves are reducible has a stratification in terms of the dual graphs. The following discussion is from \cite{behrendmanin}. In this paper we will only use the case needed in Section \ref{chains}.
\begin{defn}
	A graph $\tau$ is the data of a finite set of vertices $W$ along with a finite set of flags $F$ with map $\delta:F\rightarrow W$ and an involution $:F\rightarrow F$. The set of tails, $T$ is the set of fixed points of $j$. The set of edges is the set $(F\setminus T)/ \{j(f) \sim f\}$.
\end{defn}
\begin{defn}
	For an integer $A\ge 0$, a $A$-graph is a graph $\tau$ as above along with a map $\beta:W\rightarrow \mb{Z}_{\ge 0}$ such that $\sum_{w\in W}\beta(w)=A$.
\end{defn}
\begin{defns}
	\begin{enumerate}
		\item For a connected, at-worst-nodal curve $C$ with $n$ marked points, the associated dual graph is defined as:
			\begin{enumerate}
				\item The vertices $W$ are the set of irreducible components of $C$.
				\item The flags $F$ are inverse images of the nodes and the marked points in the normalization $\tilde{C}\rightarrow C$.
				\item The map $\delta:F\rightarrow W$ maps each point in the normalization to the irreducible component it belongs to.
				\item If $p$ and $q$ are the inverse images of a node in the normalization then $j$ maps $p$ to $q$. The marked points are mapped to themselves (hence correspond to tails).
			\end{enumerate}
		\item Let $X$ be a projective variety with $Pic(X)\cong \mb{Z}$. $f:C\rightarrow X$ is a stable map, we may associate an $A$-graph to $f$ by taking $\tau$ as above and defining $\beta(w)=\deg_{C_w} f^*L$ where $C_w$ is the irreducible component of $C$ corresponding to $w$ and $L$ is the ample generator of $Pic(X)$. Here $A=\deg_C f^*L$.
	\end{enumerate}
\end{defns}
\begin{rem}
	When $g=0$, the graphs obtained as above are always acyclic.
\end{rem}
There are natural maps between $A$-graphs that correspond to specialization of curves to curves with more irreducible components. Here is an informal definition. For more details see \cite[Definition ~1.3]{behrendmanin}.
\begin{defn}
	Given two $A$-graphs $(\tau,\beta)$ and $(\tau^{\prime},\beta^{\prime})$, a contraction $c:\tau^{\prime}\rightarrow \tau$ is a surjective map $c_W:W(\tau^{\prime})\rightarrow W(\tau)$ between the vertices that maps two adjacent vertices in $\tau^{\prime}$ to adjacent vertices in $\tau$ (every vertex is considered to be adjacent to itself). The tails in $\tau^{\prime}$ are mapped bijectively to the tails in $\tau$. Further, we must have $\beta(w)=\sum_{w^{\prime}\in c_W^{-1}(w)}\beta^{\prime}(w^{\prime})$.
\end{defn}
\begin{defn}
	For an acyclic $A$-graph $\tau$, a projective variety $X$ with $Pic(X)=\mb{Z}\langle L\rangle$, define $\mbar(X,\tau)$ to be the moduli space parameterizing stable maps from genus $0$ curves to $X$ with dual $A$-graph $\tau^{\prime}$ such that there is a contraction $c:\tau^{\prime}\rightarrow \tau$.
\end{defn}
\begin{thm}\cite[Theorem ~3.14]{behrendmanin}
	$\mbar(X,\tau)$ is a proper Deligne-Mumford stack with projective coarse moduli space.
\end{thm}
In particular, if we take $(\tau,\beta)$ to be the graph with one vertex $w$, $n$ tails and $\beta(w)=e$, we recover the stack $\mbar_{0.n}(X,e)$. For any $A$-graph $(\tau^{\prime},\beta^{\prime})$ with $A=e$ and $n$ tails, there is a natural inclusion $\mbar(X,\tau^{\prime})\rightarrow \mbar_{0,n}(X,e)$ induced by the natural contraction $\tau^{\prime}\rightarrow \tau$.

\begin{defn}\label{stablechainsdefn} Let $l$ be a positive integer, the space of two pointed $l$-chains of lines corresponds to the space $\mbar(X,\tau)$ where $\tau$ is represented by the $l$-graph below with $l$ vertices and two tails. The weight of each vertex is $1$. As discussed above, there is a natural inclusion $\mbar(X,\tau)\hookrightarrow \mbar_{0,2}(X,l)$.
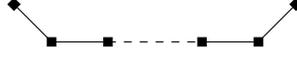
	\begin{figure}[!ht]
		\centering
		\begin{circuitikz}
			\tikzstyle{every node}=[font=\tiny]
			\node at (3.25,10.25) [squarepole] {};
			\node at (4,10.25) [squarepole] {};
			\node at (5.25,10.25) [squarepole] {};
			\node at (6,10.25) [squarepole] {};
			\draw (3.25,10.25) to[short] (4,10.25);
			\draw (5.25,10.25) to[short] (6,10.25);
			\node at (6.5,10.75) [diamondpole] {};
			\node at (2.75,10.75) [diamondpole] {};
			\draw (2.75,10.75) to[short] (3.25,10.25);
			\draw (6,10.25) to[short] (6.5,10.75);
			\draw [dashed] (4,10.25) -- (5.25,10.25);
		\end{circuitikz}

		\label{fig:my_label}
	\end{figure}
\end{defn}
In fact, assuming we are in the situation in Theorem \ref{irredlines}, we can refine the above inclusion using the following lemma:
\begin{lem}\cite[Lemma 3.5]{nk1006}\label{uniqueirredcomp}
	Assume the space of lines in $X$, $\mbar_{0,0}(X,1)$ is irreducible. Furthermore, assume the evaluation map $ev_1:\mbar_{0,1}(X,1)\rightarrow X$ has irreducible geometric generic fiber. Then for every positive interger $e$, there is a unique irreducible component, $M_{0,e}\subset \mbar_{0,0}(X,e)$ such that the special points of $M_{0,e}$ parameterize reducible curves whose non-contracted components are all finite covers of free lines in $X$. We will denote by $M_{2,e}$, the unique irreducible component of $\mbar_{0,2}(X,e)$ that dominate $M_{0,e}$.
\end{lem}
It is clear that $\mbar(X,\tau)\subset M_{2,l}$.


\section{Rationally Simply Connected}\label{sectiondefnrss}
Let $X$ be a smooth projective variety over $\mb{C}$ with Picard group $\mb{Z}$. $\mbar_{0,n}(X,e)$ denotes the space of degree $e$ (with respect to the ample generator of the Picard group) $n$-pointed stable maps from genus $0$ curves to $X$. Rational simply connectedness roughly corresponds to the space of rational curves connecting two general points being rationally connected as well. Before going to the full definition, we need the following:
\begin{defn}
	For a projective variety $Y/k$, a morphism $f:\proj{}^1_k\rightarrow Y$ is called \textit{free} (respectively \textit{very free}) if $f$ is non-constant, $f(\proj{}^1)$ is contained in the non-singular locus of $Y$ and $f^*T_Y$ is globally generated (respectively ample).
\end{defn}
\begin{prop}\cite[IV.3.7]{kol}\label{veryfreerationallyconnected}
	For a smooth projective variety $Y$, the existence of a very free rational curve $f:\proj{}^1\rightarrow Y$ is equivalent to $Y$ being rationally connected.
\end{prop}
There is a relative version of the above:
\begin{defn}
	For a morphism of projective varieties $\pi:Y\rightarrow S$, a rational curve $f:\proj{}^1\rightarrow Y$ is $\pi$-relatively \textit{free} (resp. \textit{very free}) if $f(\proj{}^1)$ lies in the smooth locus of $\pi$ and $f^*T_{\pi}$ is globally generated (resp. ample).
\end{defn}
There is an analogous result for relatively very free curves:
\begin{prop}
	For a smooth morphism of projective varieties $\pi:Y\rightarrow S$, let $s:\pone\rightarrow Y$ be a $\pi$-relatively very free curve. Then a general fiber of $\pi$ is rationally connected.
\end{prop}
\begin{proof}
	From \cite[IV.3.11]{kol}, the fiber being rationally connected is an open condition on $S$. If $s(\pone)$ is contained in a fiber of $S$ then by Proposition \ref{veryfreerationallyconnected}, it is rationally connected and we are done. Otherwise, we may consider the base-change of $\pi$ through the composite map $\pi\circ s$ and apply Lemma \ref{relativelyveryfreefibers}.
\end{proof}
Let $X$ be a smooth projective variety with $Pic (X) \cong \mb{Z}$. Let $ev:\mbar_{0,1}(X,1)\rightarrow X$ denote the space of $1$-pointed rational curves in $X$ with degree $1$ (with respect to the ample generator of the Picard group) along with the natural evaluation map. Similarly $\mbar_{0,0}(X,e)$ denotes the stack of genus $0$ stable maps of degree $e$. The following notions are a slight modification to the definitions in \cite{twistabledefn}.
\begin{defn}\label{twistabledef} Consider a curve $\zeta \in \mbar_{0,0}(X,e)$ corresponding to the map $a_{\zeta}:C\rightarrow X$. $\zeta$ is called \textit{twistable} (respectively \textit{very twistable}) if:
	\begin{enumerate}
		\item $[\zeta]$ is a free curve in $X$.
		\item There is a section $s$ of $ev:\mbar_{0,1}(X,1)\prescript{}{ev}{\times}_{a_{\zeta}} C\rightarrow C$ corresponding to a family of pointed lines $(\pi:\Sigma \rightarrow C, \sigma, h)$.
		\item Points of $s(C)$ parametrizes unobstructed lines in $X$.
		\item $s$ is $ev$-relatively free (respectively very free)
		\item $\deg \mc{N}_{\sigma(C)/\Sigma}\ge 0$ where $\mc{N}_{\sigma(C)/\Sigma}$ is the normal bundle of the section $\sigma$.
	\end{enumerate}
\end{defn}
\begin{rem}\label{normalbundledefinedglobally}
	Observe that there is a universal family $\pi:\mc{C}\rightarrow \mbar_{0,1}(X,1), \sigma:\mbar_{0,1}(X,1)\rightarrow \mc{C}$ of $1$-pointed lines. Thus the normal bundle of the section $\mc{N}_{\sigma}$ is a globally defined line bundle on $\mbar_{0,1}(X,1)$
\end{rem}
There is a related notion of $\textit{twisting}$ surfaces from \cite{nk1006} which we now state.
\begin{defn}\label{twistingsurfacedefn}
	Let $n$ be a non-negative integer. Let $\pi:\Sigma \rightarrow \pone, f:\Sigma \rightarrow X$ be a family of rational curves in $X$ where $\Sigma$ is a Hirzebruch surface of type $h$. Let $F$ denote a general fiber of $\pi$ and $F^{\prime}$ denote an irreducible divisor on $\Sigma$ with $(F^{\prime}.F^{\prime})=h$. We say the family above is an \textit{$n$-twisting} surface in $X$ if:
	\begin{enumerate}
		\item $f^*T_X$ is a globally generated vector bundle on $\Sigma$.
		\item The map $(f,\pi):\Sigma\rightarrow X\times \pone$ is finite.
		\item The twisted normal bundle $N_{(f,\pi)}(-F^{\prime}-nF)$ has vanishing first cohomology.
	\end{enumerate}
	If $\beta^{\prime} = [f(F^{\prime})]$ and $\beta=[f(F)]$ are the curve classes in $X$, then we say the twisting surface has class $(\beta,\beta^{\prime})$.
\end{defn}
These two notions are indeed related. We have the following lemma:
\begin{lem}
	Let $a:\pone \rightarrow X$ be a twistable (resp. very twistable) curve. Let $(\pi:\Sigma \rightarrow C, \sigma, h)$ be the corresponding family of lines in $X$. Furthermore, assume $\Sigma \cong \pone\times\pone$ with $\pi$ being one of the projections. Then $(\pi:\Sigma\rightarrow \pone,h:\Sigma\rightarrow X)$ is a $1$-twisting (resp. $2$-twisting) surface.
\end{lem}
\begin{proof}
	To check that $h^*T_X$ is globally generated, note that we have the following exact sequence:
	$$0\rightarrow T_{\Sigma}\rightarrow h^*T_X\oplus \pi^*T_{\pone}\rightarrow N_{(h,\pi)}\rightarrow 0$$
	Since $T_{\Sigma}$ is globally generated, it is enough to check that $N_{(h,\pi)}$ is globally generated. Pulling back the above sequence via $\sigma$, we observe that $\sigma^*h^*T_X\oplus T_{\pone}$ surjects onto $\sigma^*N_{(h,\pi)}$. Since $a=h\cdot\sigma$ is a free curve, $\sigma^*N_{(h,\pi)}$ is globally generated. We have the following exact sequence on $\Sigma$:
	$$0\rightarrow N_{(h,\pi)}(-\sigma(\pone))\rightarrow N_{(h,\pi)}\rightarrow N_{(h,\pi)}\vert_{\sigma(\pone)}\rightarrow 0$$
	Since fibers of $\pi$ parametrizes free lines in $X$, we have: $h^1(F,N_{(\pi,h)}(-\sigma(\pone))\vert_F)=0$ where $F$ is any fiber of $\pi$. Thus $R^1\pi_*(N_{(\pi,h)}(-\sigma(\pone))) =0$. So we can push-forward the above exact sequence via $\pi$ to get:
	$$0\rightarrow \pi_* N_{(h,\pi)}(-\sigma(\pone))\rightarrow \pi_* N_{(h,\pi)}\rightarrow \sigma^*N_{(h,\pi)}\rightarrow 0$$
	From Lemma \ref{descoftev}, we have $\pi_* N_{(h,\pi)}(-\sigma(\pone)$ is globally generated. Hence $\pi_*N_{(h,\pi)}$ is globally generated. Moreover, since $N_{(h,\pi)}\vert_F$ is globally generated for any fiber $F$ of $\pi$, we have that $N_{(h,\pi)}$ is globally generated.\\
	Condition $(2)$ is immediate. For condition $(3)$, observe that, since $\deg \mc{N}_{\sigma(\pone)/\Sigma}\ge 0$, we have: $\sigma(\pone)=F^{\prime}+mF$ for some $m\ge 0$. Moreover, by assumption and Lemma \ref{descoftev}, $\pi_*N_{(h,\pi)}(-\sigma(\pone))$ is globally generated and hence $h^1(\pone,\pi_*N_{(h,\pi)}(-\sigma(\pone)-F))=0$. As before, we also have: $R^1\pi_*N_{(h,\pi)}(-\sigma(\pone)-F)=0$. Thus, from the Leray spectral sequence, we get: $H^1(\Sigma,N_{(h,\pi)}(-\sigma(\pone)-F))=0\implies H^1(\Sigma,N_{(h,\pi)}(-F^{\prime}-(m+1)F)=0$ for some $m\ge0$. Thus $\Sigma$ is $(m+1)$-twisting $\implies$ $\Sigma$ is $1$-twisting. The argument for very twistable implies $2$-twisting is very similar.
\end{proof}
\begin{rem}\label{classoftwistingsurface}
	Note that if $\sigma(\pone)=F^{\prime}$ in the above, the class of the twisting surface is $([L], [a(\pone)])$ where $[L]$ is the class of a line.
\end{rem}
For convenience, we include the following two lemmas that will be necessary later on.
\begin{lem}\label{deformtwistable}\cite[Corollary 7.7(i)]{nk1006}
	Let $f:\Sigma \rightarrow X$ be an $n$-twisting surface with $F^{\prime}$ and $F$ as in Definition \ref{twistingsurfacedefn}. Assume $\Sigma \cong \pone\times\pone$. Let $D \equiv F^{\prime} + nF$ be a reduced divisor. Then any small deformation $(D_1,g)$ of $(D, f\vert_D)$ in $\mbar_{0,0}(X)$ is contained in an $n$-twisting surface $f_1:\Sigma_1 \rightarrow X$ with $D_1\equiv F^{\prime}+nF$ and $g=f_1\vert_{D_1}$. Moreover, we can take $\Sigma_1\cong \pone\times\pone$.
\end{lem}
\begin{lem}\label{uniontwistable}\cite[Lemma 7.6(v)]{nk1006}
	For $i=1,2$, let $f_i:\Sigma_{i}\rightarrow X, \pi_i:\Sigma_i\rightarrow \pone$ be $n_i$-twisting surfaces with $n_i\ge 0$ of class $(\beta, \beta_i)$. Assume $\Sigma_i\cong \pone\times\pone$. Moreover, let $F_i$ be fibers of $\pi_i$ such that: $$f_1\vert_{F_1}=f_2\vert_{F_2}$$
	Then there exists an $n_1+n_2-1$-twisting surface in $X$ isomorphic to $\pone\times\pone$ and of class $(\beta,\beta_1+\beta_2)$
\end{lem}
Combining the two lemmas above, we have:
\begin{cor}\label{addtwistables}
	Let $f_i:\Sigma_{i} \cong \pone\times\pone \rightarrow X$ be two $1$-twisting surfaces of classes $(\beta,\beta_1)$ and $(\beta,\beta_2)$ where $\beta = [L]$ is the class of a line in $X$. Moreover, assume that the space of lines through a general point of $X$ is irreducible. Then there are $1$-twisting surfaces of classes $(\beta,a_1\beta_1+a_2\beta_2)$ for any non-negative integers $a_1$ and $a_2$.
\end{cor}

Now we are ready to define rationally simply connectedness.
\begin{defn}\label{rss}
	Let $ev_1:\mbar_{0,1}(X,1)\rightarrow X$ and $ev_{2,e}:\mbar_{0,2}(X,e)\rightarrow X\times X$ denote the respective evaluation morphisms. Then $X$ is rationally simply connected if:
	\begin{enumerate}
		\item A general fiber of $ev_1$ is irreducible, rationally connected and non-singular.
		\item For large enough $e>0$ there are irreducible components $M_{2,e}\subset \mbar_{0,2}(X,e)$ such that $ev_{2,e}|_{M_{2,e}}$ is dominant and has rationally connected geometric generic fiber.
		\item There is a very twistable curve $a:\proj{}^1\rightarrow X$
	\end{enumerate}
\end{defn}

\begin{hyp}\label{rsshyp}
	$X\subset OG=OG(k,V)$ (where $\dim V=n+1$) is a general degree $d$ hypersurface satisfying the following inequality:
	$$n+1-8k-4\ge 3kd^2-d^2-d$$
\end{hyp}
We will state the main theorem here:
\begin{thm}\label{mainthm}
	Let $V$ be a $n+1$-dimensional vector space with a symmetric, non-degenerate bilinear form. Let $OG(k,V)$ denote the orthogonal grassmannian parameterizing $k$-dimensional isotropic subspaces of $V$. Let $X\subset OG(k,V)$ be a general degree $d$ hypersuface satisfying Assumption \ref{rsshyp}. Then $X$ is rationally simply connected.
\end{thm}


\section{Twistable Curves}\label{sectiontwistable}

In this section we will discuss the existence of particular twistable curves in a hypersurface $X$ in $OG$. We will assume we are in the situation of Assumption \ref{rsshyp} but some of the statements will be more generally true.

\subsection{Twisting Families in the Orthogonal Grassmannian}
Let $s:\proj{}^1\rightarrow \mbar_{0,1}(OG)$ be a family of lines in $OG$ defined by $(\pi:\Sigma \rightarrow \proj{}^1, h:\Sigma \rightarrow X, \sigma: \proj{}^1\rightarrow \Sigma)$. Using the universal property, we get three isotropic subbundles of $V\otimes \mc{O}_{\proj{}^1}$: $S_{k-1}\subset S_k \subset S_{k+1}$ of rank $(k-1),k$ and $(k+1)$ such that $\Sigma=\proj{S_{k+1}/S_{k-1}}$ and $\sigma$ is given by $S_k/S_{k-1}\subset S_{k+1}/S_{k-1}$. Then:
\begin{prop}\cite[Corollary ~3.4]{twistingfamilies} Using the above notation:
	\begin{enumerate}
		\item $s^*T_{ev}\cong ((S_{k+1}/S_k)^*\otimes (S_{k+1}^{\perp}/S_{k+1}))\oplus ((S_k/S_{k-1})\otimes S_{k-1}^*)$
		\item $N_{\sigma}\cong (S_{k+1}/S_k)\otimes (S_k/S_{k-1})^*$
	\end{enumerate}
\end{prop}

\begin{rem}
	All rational curves in a projective homogenous space are unobstructed. Thus we only need to worry about the positivity of the above two bundles.
\end{rem}

The following types of twistable curves exist in $OG(k,n+1)$:
\begin{prop}\label{myfamilies}
	Under Assumption \ref{rsshyp}, there exists the following curves in $OG$:
	\begin{enumerate}
		\item twistable curves $[\zeta_{i}]\in \mbar_{0,0}(OG,e_i)$ for $i=0,\ldots,k-1$ where $e_i=k+i$.
		\item very twistable curve $[\zeta]\in \mbar_{0,0}(OG,e)$ where $e=3k-1$
	\end{enumerate}
\end{prop}
\begin{proof}
	The corresponding families are given by the following collection of vector bundles.
	\begin{enumerate}
		\item For $i=0,\ldots,k-1$ take: $$S_{k-1}=\mc{O}(-1)^{k-1-i}\oplus\mc{O}(-2)^{i}; S_k=S_{k-1}\oplus \mc{O}(-1); S_{k+1}=S_k\oplus \mc{O}(-1)$$
		\item $S_{k-1}=\mc{O}(-3)^{k-1}; S_k=S_{k-1}\oplus \mc{O}(-2); S_{k+1}=S_k\oplus \mc{O}(-2)$

	\end{enumerate}
	In all the above cases, we choose the bundles such that the map $V^*\otimes\mc{O}_{\pone}\rightarrow S_{k+1}^*$ is surjective on global sections. This can be done, for example, by choosing a maximal isotropic subspace $W\subset V$ and taking $S_{k+1}^*$ to be a quotient of $W^*\otimes\mc{O}_{\pone}$ where the quotient map is surjective on global sections. It is easy to see that this is possible under Assumption \ref{rsshyp}. Hence $S_{k+1}^{\perp}/S_{k+1}$ has $\mc{O}(-1)$ as the most negative summand. This gives us the required positivity of $s^*T_{ev}$ in each of the cases. Computing the degrees of $S_k$ for every family, we get: $e_i=k+i$ and $e=3k-1$.
\end{proof}

\subsection{Twisting Families in Hypersurfaces}
Let $X \subset OG$ denote a general degree $d$, Fano hypersurface. In this section we will show that the twistable curves in Proposition \ref{myfamilies} are also twistable in $X$. First we need an openness condition.
\begin{lem}
	Suppose $[\epsilon]\in \mbar_{0,0}(X_0,e)$ is an irreducible twistable (resp. very twistable) curve in a degree $d$ hypersurface $X_0$ in $OG$ with associated twisting (resp. very twisting) family $\zeta:\proj{}^1\rightarrow \mbar_{0,1}(X_0,1)$. Then we can also find a twistable (resp. very twistable) curve of degree $e$ in a general degree $d$ hypersuface $X\subset OG$.
\end{lem}
\begin{proof}
	Let $V_d=\Gamma(OG,\mc{O}(d))$ and let $\pi:\mc{X}\rightarrow \proj{V_d}$ denote the universal family of hypersurfaces. Denote by $\mbar_{0,1}(\mc{X},1)$ the space parameterzing $(p\in l\subset X)$ where $p$ is a point on a line $l$ contained in a hypersurface $\mc{X}$. There is an evaluation map $ev:\mbar_{0,1}(\mc{X},1)\rightarrow \mc{X}$. We know that a general fiber of $ev$ is non-singular and rationally connected. By assumption, the $\zeta(\proj{}^1)$ lies in the smooth locus of $ev$. There is also a forgetful map $\phi:\mbar_{0,1}(\mc{X},1)\rightarrow \mbar_{0,1}(OG,1)$. Let $(p_0:\mc{C} \rightarrow \mbar_{0,1}(OG,1), \sigma:\mbar_{0,1}(OG,1) \rightarrow \mc{C})$ be the universal family and section. Then the normal bundle of $\sigma$, $\mc{N}_{\sigma}$ is a line bundle on $\mbar_{0,1}(OG,1)$. On $\mbar_{0,1}(\mc{X},1)$, consider the following lines bundles:
	$$\mc{L}_1=\phi^*\mc{N}_{\sigma},\qquad \mc{L}_2=ev^*\pi^*\mc{O}_{\proj{V_d}}(1)\qquad
	\mc{L}_3=\phi^*ev^*\mc{O}_{OG}(1)$$
	Then, there is an irreducible component $M$ of $Mor(\proj{}^1,\mbar_{0,1}(\mc{X},1))$ containing $[\zeta]$ such that for a general $[\zeta^{\prime}]\in M$, $\zeta^{\prime}$ has irreducible domain, $\deg (\zeta^{\prime *} \mc{L}_1)\ge 0$, $\deg (\zeta^{\prime *}\mc{L}_2)=0$ and $\deg (\zeta^{\prime *} \mc{L}_3) = e$. Thus $\zeta^{\prime}(\proj{}^1)$ lies in a fiber of $ev\circ \pi$. Moreover, since $\zeta^*T_{ev}$ is globally generated, $\zeta$ may be deformed to contain a general point of $\mbar_{0,1}(\mc{X},1)$. Therefore, a general such deformation $\zeta^{\prime}$ lies in $\mbar_{0,1}(X,1)=(ev\circ \pi)^{-1}([X])$ for a general hypersurface $X$ and is twistable. The proof for very twistable is similar.
\end{proof}
Thus it is enough to find one hypersurface with a twistable (or very twistable) curve of each degree. Let $(\pi:\Sigma\rightarrow \proj{}^1,\sigma,h)$ be a family of $1$-pointed lines in $OG$. Let $X$ be a degree $d$ hypersurface containing $h(\Sigma)$. For $Y=X$ and $Y=OG$, let $ev_Y:\mbar_{0,1}(Y,1)\rightarrow Y$ denote the evaluation map. Furthermore, denote by $N_{\Sigma/Y\times \proj{}^1}$, the normal bundle of the product map $(h,\pi):\Sigma \rightarrow Y\times \proj{}^1$. Let $s:\proj{}^1\rightarrow \mbar_{0,1}(X,1)\subset \mbar_{0,1}(OG,1)$ denote the map corresponding to the above family. Then we have:
\begin{lem}\cite[Proposition ~5.7]{deland}\label{descoftev}
	For each choice of $Y$, assuming the lines parametrized by $s$ are unobstructed in $Y$, there is a natural isomorphism:
	$$s^*T_{ev_Y} \cong \pi_*N_{\Sigma/Y\times \proj{}^1}(-\sigma(\proj{}^1))$$
\end{lem}

On $\Sigma$, there is a short exact sequence of locally free sheaves:
$$0\rightarrow N_{\Sigma/X\times \proj{}^1}(-\sigma(\proj{}^1))\rightarrow N_{\Sigma/OG\times \proj{}^1}(-\sigma(\proj{}^1))\xrightarrow{d_X} h^*N_{X/OG} (-\sigma(\proj{}^1))\rightarrow 0$$
Where $d_X$ is the map induced by the Jacobian of $X$. Let $\mc{O}(a)$ be a line bundle on $\proj{}^1$. Let $L$ on $\Sigma$ denote the bundle $\mc{O}_{\Sigma}(-\sigma(\proj{}^1))\otimes \pi^*\mc{O}_{\proj{}^1}(-a)$. Tensor and take global sections to get the left exact sequence:
$$0\rightarrow H^0(\Sigma,N_{\Sigma/X\times \proj{}^1}\otimes L)\rightarrow H^0(\Sigma,N_{\Sigma/OG\times \proj{}^1}\otimes L)\xrightarrow{D^L_X} H^0(\Sigma,h^*N_{X/OG} \otimes L)$$
\begin{lem}\label{enoughtosurject}
	Assume $H^1(\pone, \pi_*(N_{\Sigma/OG\times \proj{}^1}\otimes L))=0$ and $\pi_*(h^*N_{X/OG}\otimes L)$ is globally generated. Furthermore, assume $D^L_X$ is surjective. Let $F\subset X$ be any line parametrized by $s(\pone)$. Then:
	\begin{enumerate}
		\item $H^1(\pone, \pi_*(N_{\Sigma/X\times \proj{}^1}\otimes L))=0$
		\item $H^1(F, N_{F/X}(-1))=0$ i.e. $F$ is unobstructed in $X$.
	\end{enumerate}
\end{lem}
\begin{proof} Observe that for $Y=X$ or $OG$, $N_{F/Y}(-1)=(N_{\Sigma/Y\times \proj{}^1}\otimes L)|_F$
	Since all lines in $OG$ are unobstructed, $H^1(F, N_{F/OG}(-1))=0$ for all fibers of $\pi$. Hence the derived push-forward, $R^1\pi_*(N_{\Sigma/OG\times \proj{}^1}\otimes L))=0$. Thus we have a surjective map of vector bundles:
	$$\psi:\pi_*(h^*N_{X/OG}\otimes L) \rightarrow R^1\pi_*(N_{\Sigma/X\times \proj{}^1}\otimes L)$$
	However, since the map on global sections $H^0(\pone,\pi_*(N_{\Sigma/OG\times \proj{}^1}\otimes L))\xrightarrow{\pi_*D^L_X} H^0(\pone,\pi_*(h^*N_{X/OG} \otimes L))$ is surjective, $\psi$ vanishes on global sections. Since $\pi_*(h^*N_{X/OG}\otimes L)$ is globally generated, we have $R^1\pi_*(N_{\Sigma/X\times \proj{}^1}\otimes L)=0$. Since $\pi$ is flat, we have $(2)$ from the cohomology and base change theorem (see \cite[III.12.11]{hartshorne}). Moreover, we have a short exact sequence:
	$$0\rightarrow \pi_*(N_{\Sigma/X\times \proj{}^1}\otimes L)\rightarrow \pi_*(N_{\Sigma/OG\times \proj{}^1}\otimes L)\rightarrow \pi_*(h^*N_{X/OG}\otimes L)\rightarrow 0$$ Taking the long exact sequence in cohomology, we have $(1)$. 
\end{proof}
\begin{rem}\label{enoughttoshowsurjective}
	There is a natural map $h^*T_{OG}\rightarrow N_{\Sigma/OG\times \pone}$. Additionally, we may identify $h^*N_{X/G}=h^*\mc{O}_{OG}(d)$. Thus to show that $D^L_X$ is surjective, it is enough to show that the following derivative map induced by $X$ is surjective:
	$$H^0(h^*T_{OG}\otimes L)\xrightarrow{D_{X,L}} H^0(h^*\mc{O}_{OG}(d)\otimes L)$$ 
\end{rem}
\begin{hyp}
	For the rest of the section, we will make the following assumptions:
	\begin{enumerate}
		\item $(\pi:\Sigma\rightarrow \proj{}^1,\sigma,h)$ as above is given by a trio of vector bundles $S_{k-1}\subset S_k\subset S_{k+1}$ such that $V^*$ surjects onto the global sections of $S_{k+1}^*, S_k^*, S_{k-1}^*$. Thus $\Sigma=\proj{S_{k+1}/S_{k-1}}$.
		\item $L=\mc{O}_{\Sigma}(-1)$.
	\end{enumerate}
\end{hyp}
\begin{lem}\label{constantsubbundle}
	There is a constant subbundle $W\otimes \mc{O}_{\Sigma}\subset h^*T_{OG}\otimes L$ such that $\dim W\ge n+1-2\dim H^0(\pone, S_{k+1}^*)$.
\end{lem}
\begin{proof}
	On $\pone$, observe that $(V/S_{k+1})^*=Ker(V^*\otimes \mc{O}_{\pone}\rightarrow S_{k+1}^*)$ has $\mc{O}_{\pone}$ and $\mc{O}(-1)_{\pone}$ as the only line bundles appearing in the direct sum decomposition. Let $W^{\prime}=Ker(V^*\rightarrow H^0(\pone,S_{k+1}^*))$. Therefore $W^{\prime}\otimes \mc{O}_{\pone}$ is a direct summand of $(V/S_{k+1})^*$. Note that $\dim W^{\prime}=n+1-\dim H^0(\pone,S_{k+1}^*)$. Since the map $(V/S_{k+1})^*\rightarrow (V/S_{k-1})^*$ is injective, $W^{\prime}\otimes \mc{O}_{\pone}$ is also a direct summand of $(V/S_{k-1})^*$. Let $W_1=W^{\prime *}$. Then $W_1\otimes \mc{O}_{\pone}$ is a direct summand of $V/S_{k-1}$ and $V/S_{k+1}$. Let $E_k$ on $\Sigma$ denote the pull-back of the tautological bundle from $OG$. By construction, we have: $\pi^*S_{k-1}\subset E_k\subset \pi^*S_{k+1}$. There are associated maps: $\pi^*V/S_{k-1}\rightarrow V/E_k\rightarrow \pi^*V/S_{k+1}$. Thus $W_1\otimes \mc{O}_{\Sigma}$ is a subbundle of $V/E_k$. There is an inclusion:
	$\mc{O}_{\Sigma}(1)\subset E_k^*$. Tensoring with $W_1\otimes L\subset V/E_k\otimes L$, we get an inclusion: $W_1\otimes \mc{O}_{\Sigma}\subset (V/E_k)\otimes E_k^*\otimes L$. There is a short exact sequence on $\Sigma$:
	$$0\rightarrow T_{OG}\otimes L\rightarrow V/E_k\otimes E_k^*\otimes L\rightarrow Sym^2(E_k^*)\otimes L\rightarrow 0$$
	$Sym^2(E_k^*)\otimes L$ has a filtration whose grades parts are: $Sym^2(\pi^*S_{k-1}^*)\otimes L, \pi^*(S_{k-1}^*), \mc{O}_{\Sigma}(2)\otimes L=\mc{O}_{\Sigma}(1)$. $Sym^2(\pi^*S_{k-1}^*)\otimes L$ has no global sections since its degree on the fibers of $\pi$ is $-1$. $H^0(\Sigma,\mc{O}_{\Sigma}(1))=H^0(\pone,(S_{k+1}/S_{k-1})^*)$ and $H^0(\Sigma,\pi^*S_{k-1}^*)=H^0(\pone,S_{k-1}^*)$. Thus $\dim H^0(\Sigma,Sym^2(E_k^*))\le \dim H^0(\pone,S_{k+1}^*)$. Finally, take $W=Ker(W_1\rightarrow H^0(\Sigma,Sym^2(E_k^*)))$. Clearly, $W\otimes \mc{O}_{\Sigma} \subset T_{OG}\otimes L$ and $\dim W \ge n+1-2\dim H^0(\pone, S_{k+1}^*)$.
\end{proof}
\begin{rem}\label{linearvanishing}
	Recall that we have the Plucker embedding $OG(k,n+1)\subset \proj{\bigwedge^{k}V}$. Under this embedding, we have: $\bigwedge^{k}V^*= H^0(\proj{\bigwedge^{k}V},\mc{O}(1))=H^0(OG,\mc{O}(1))$. In particular, using the notation from the above lemma, $Im(W^*\otimes \bigwedge^{k-1}V^*)\subset \bigwedge^{k}V^*$ vanishes on $\Sigma$. Let $V_d$ denote $H^0(OG,\mc{O}(d))$ and let $W^d$ denote the subspace of $V_d$ vanishing on $\Sigma$. There is a natural map $W^*\otimes \bigwedge^{k-1}V^*\otimes V_{d-1}\rightarrow V_d$ whose image is contained in $W^d$.
\end{rem}

There is a universal derivative map on $\Sigma$, $D:W^d\otimes T_{OG} \otimes L \rightarrow h^*\mc{O}_{OG}(d)\otimes L$. Restricting to the subbundle from Lemma \ref{constantsubbundle} and using Remark \ref{linearvanishing}, we get the following map:
$$D_L: W^*\otimes \bigwedge^{k-1}V^*\otimes V_{d-1}\otimes W\otimes \mc{O}_{\Sigma}\rightarrow h^*\mc{O}_{OG}(d)\otimes L=h^*\mc{O}_{OG}(d-1)\otimes \pi^*(\bigwedge^{k-1} S_{k-1}^*)$$
(since $\mc{O}_{\Sigma}(-1)=h^*\mc{O}_{OG}(-1)\otimes \pi^*(\bigwedge^{k-1} S_{k-1}^*)$).
\begin{lem}\label{descriptionofderivativemap}
	Let $e:W\otimes W^* \otimes \mc{O}_{\Sigma}\rightarrow \mc{O}_{\Sigma}$ be the natural evaluation map. Let $\rho^{k-1}:\bigwedge^{k-1}V^*\otimes \mc{O}_{\Sigma}\rightarrow \pi^*(\bigwedge^{k-1} S_{k-1}^*)$ be the wedge product of the map $V^*\otimes \mc{O}_{\pone}\rightarrow S_{k-1}^*$. Let $r_{d-1}:V_{d-1}\otimes \mc{O}_{\Sigma}\rightarrow h^*\mc{O}_{OG}(d-1)$ be the restriction map. Then:
	$$D_L=e\otimes \rho^{k-1}\otimes r_{d-1}$$
\end{lem}
\begin{proof}
	First observe that it is enough to check this for $d=1$. Since $Im(W^*\otimes \bigwedge^{k-1}V^*)$ vanishes on $\Sigma$, the statement for $d>1$ follows by the Leibniz rule. We can check the equality of the maps pointwise. Let $P\in \Sigma$ correspond to a chain of vector-spaces $U_{k-1}\subset U_{k}\subset U_{k+1}$ of dimensions $k-1,k,k+1$. The map $D_L$ is induced by the inclusion $W\otimes (U_k/U_{k-1})^* \hookrightarrow V/U_k\otimes U_k^*=Hom(U_k,V/U_k)$. We may extend a map $t:U_k\rightarrow V/U_k$ to $\tilde{t}:U_k\rightarrow V$ by chosing a splitting $V=V/U_k\oplus U_k$ and using the identity map $U_k\rightarrow U_k$. The derivative map is simply pulling back the forms through the map $U_k\rightarrow V$. Since $W^*$ vanishes on $U_k$, the pull-back of $w^*\in W^*$ via $\tilde{t}$ factors through $t$. Thus if $t\in W\otimes (U_k/U_{k-1})^*$, the pull-back of $w^*$ is simply evaluation on $W$. Hence the derivative map $W\otimes (U_k/U_{k-1})^*\otimes W^*\otimes \bigwedge^{k-1}V^*\rightarrow \bigwedge^{k}U_k^*$ factors through the evaluation map $e|_P$. We are left with the map: $ (U_k/U_{k-1})^*\otimes \bigwedge^{k-1}V^*\rightarrow \bigwedge^{k}U_k^*$. Tensoring with $(U_k/U_{k-1})=\bigwedge^k U_k\otimes \bigwedge^{k-1}U_{k-1}^*$, we get our desired map: $D_L: \bigwedge^k U_k^*\otimes \bigwedge^{k-1}U_{k-1}\otimes \bigwedge^{k-1}V^*\otimes \bigwedge^k U_k\otimes \bigwedge^{k-1}U_{k-1}^*\rightarrow \bigwedge^{k-1}U_{k-1}^*$. A priori $D_L$ is the evaluation map $\bigwedge^{k-1}U_{k-1}\otimes \bigwedge^{k-1}V^*\rightarrow T$ (where $T$ is the trivial $1$-dimensional vector space) followed by the pairing $ \bigwedge^k U_k^*\otimes \bigwedge^k U_k\otimes \bigwedge^{k-1}U_{k-1}^*\rightarrow \bigwedge^{k-1}U_{k-1}^*$. However it is easy to check that this is the same map as the evaluation $\bigwedge^k U_k^*\otimes\bigwedge^{k-1}U_{k-1}\otimes \bigwedge^k U_k\otimes \bigwedge^{k-1}U_{k-1}^*\rightarrow T$ followed by the restriction map $\rho^{k-1}|_P$. Thus $D_L=e\otimes \rho^{k-1}$ when $d=1$. 
\end{proof}

\begin{lem}
	The restriction maps $V_d\rightarrow H^0(\Sigma, h^*\mc{O}_{OG}(d))$ is surjective for all $d\ge 1$.
\end{lem}
\begin{proof}
	Observe that the maps $Sym^d(V_1)\rightarrow V_d$ are surjective. Let $E_k$ denote the pull-back of the tautological bundle from $OG$. From the short exact sequence $0\rightarrow \pi^*S_{k-1}\rightarrow E_k \rightarrow \mc{O}_E(1)\rightarrow 0$, we see that, $h^*\mc{O}_{OG}(d)=\mc{O}_{\Sigma}(d)\otimes \pi^* (\bigwedge^{k-1}S_{k-1}^*)^{\otimes d}$. Pushing down via $\pi$, $H^0(\Sigma, h^*\mc{O}_{OG}(d))=H^0(\pone,Sym^d((S_{k+1}/S_{k-1})^*\otimes \bigwedge^{k-1}S_{k-1}^*))$. Therefore, $Sym^d(H^0(\Sigma, h^*\mc{O}_{OG}(1)))$ surjects onto $H^0(\Sigma,h^*\mc{O}_{OG}(d))$. Thus it is enough to prove the lemma for $d=1$. Recall that $V_1=\bigwedge^{k}V^*$ and $H^0(\Sigma, h^*\mc{O}_{OG}(1))=H^0(\pone,(S_{k+1}/S_{k-1})^*\otimes \bigwedge^{k-1}S_{k-1}^*)$. Let $K^*\subset V^*$ be the kernel of the map $V^*\rightarrow H^0(\pone,S_{k-1}^*)$. Then $K^*$ surjects onto $H^0(\pone,(S_{k+1}/S_{k-1})^*)$ (considered as a subspace of $H^0(\pone,S_{k+1}^*)$). Thus, $K^*\otimes \bigwedge^{k-1}V^*$ surjects onto $H^0(\pone,(S_{k+1}/S_{k-1})^*)\otimes H^0(\pone,\bigwedge^{k-1}S_{k-1}^*)$ which surjects onto $H^0(\pone,(S_{k+1}/S_{k-1})^*\otimes \bigwedge^{k-1}S_{k-1}^*)$. This map factors through $K^*\otimes \bigwedge^{k-1}V^*\rightarrow \bigwedge^k V^*$. Thus $V_1\rightarrow H^0(\Sigma, h^*\mc{O}_{OG}(1))$ is surjective.
\end{proof}
\begin{cor}\label{inequalityweneed}
	If $\dim H^0(h^*\mc{O}_{OG}(d-1)\otimes \pi^*(\bigwedge^{k-1} S_{k-1}^*)) \le n+1-2\dim H^0(\pone, S_{k+1}^*)$, there is a degree $d$ hypersurface $X$ containing $\Sigma$, such that the map $D^L_X$ (from Lemma \ref{enoughtosurject}) is surjective.
\end{cor}
\begin{proof}
	Let $\{a_i\otimes b_i\lvert 1\le i \le N\}$ be elements in $H^0(\Sigma,h^*\mc{O}_{OG}(d-1))\otimes H^0(\Sigma,\pi^*(\bigwedge^{k-1} S_{k-1}^*))$ whose images give a basis for $H^0(h^*\mc{O}_{OG}(d-1)\otimes \pi^*(\bigwedge^{k-1} S_{k-1}^*))$. Chose $\overline{b_i}\in \bigwedge^{k-1}(V^*)$ to be elements that map to $b_i$ under $\rho^{k-1}$. Chose $\overline{a_i}$ to be inverse images of $a_i$ under $r_{d-1}$. And chose linearly independent elements $x_1,\ldots,x_N\in W^*$. Let the equation of $X$ be given by the image of $F=\sum_{i=1}^N x_i\otimes \overline{b_i}\otimes\overline{a_i}$ in $V_d$. Clearly $F$ vanishes on $\Sigma$. Then the induced map $D_F: W\otimes \mc{O}_{\Sigma}\rightarrow h^*\mc{O}_{OG}(d)\otimes L$ is surjective on global sections following Lemma \ref{descriptionofderivativemap}. Thus the map $D_{X,L}:H^0(\Sigma,T_{OG}\otimes L)\rightarrow H^0(\Sigma,h^*\mc{O}_{OG}(d)\otimes L)$ is also surjective. From Remark \ref{enoughttoshowsurjective}, we have the desired result.
\end{proof}
\begin{rem}\label{nonsingularalongsigma}
	Note that the bundle $h^*\mc{O}_{OG}(d)\otimes L$ is globally generated on $\Sigma$. Thus the derivative map $T_{OG}\rightarrow N_{X/OG}$ is surjective everywhere on $\Sigma$. Therefore $X$ is non-singular along $\Sigma$. 
\end{rem}
Finally, we can discuss the families in $X$. Consider the families from Proposition $1.7$. The degree of the normal bundle of the section $\sigma$ is independent of the ambient space. Let $L_i=\mc{O}_{\Sigma}(-\sigma(\pone))\otimes \pi^*\mc{O}_{\pone}(-i)$ for $i=1,2$. For our families to be twisting (respectively very twisting) in $X$, we need $T_{ev_X}(-1)$ (respectively $T_{ev_X}(-2)$) to have vanishing higher cohomology. 
\begin{lem}\label{descoftwistinglb}
	For the families in Proposition \ref{myfamilies}: $L_1 = \mc{O}_{\Sigma}(-1)$ for the families in $(1)$ and $L_2=\mc{O}_{\Sigma}(-1)$ for family $(2)$. Moreover $\pi_*(h^*\mc{O}_{OG}(d)\otimes \mc{O}_{\Sigma}(-1))$ is globally generated in all cases.
\end{lem}
\begin{proof}
	By pulling back the relevant bundles through $\sigma$, we can see: $\mc{O}(-\sigma(\pone))=\mc{O}_{\Sigma}(-1)\otimes \pi^*(S_{k+1}/S_k)^*$. The first claim follows from the explicit description of the families.\\
	For the second claim, note that $h^*\mc{O}_{OG}(1)=\mc{O}_{\Sigma}(1)\otimes \pi^*(\bigwedge^{k-1}S_{k-1}^*)$. Thus $\pi_*(h^*\mc{O}_{OG}(d)\otimes \mc{O}_{\Sigma}(-1))=Sym^{d-1}((S_{k+1}/S_{k-1})^*)\otimes (\bigwedge^{k-1}S_{k-1}^*)^{d}$. This is globally generated in all of our cases.
\end{proof}
Thus, from Remark \ref{enoughttoshowsurjective}, Lemma \ref{enoughtosurject} and Lemma \ref{descoftev}, it is enough to find $X$ such that $D_{X,L_1}$ is surjective for families in $(1)$ (respectively $D_{X,L_2}$ is surjective for family $(2)$). 
\begin{prop}\label{twistablesinx}
	Under Assumption \ref{rsshyp}, for a general degree $d$ hypersurface $X$ in $OG$, $X$ contains the type of twistable and very twistable curves constructed in Proposition \ref{myfamilies}.
\end{prop}
\begin{proof}
	We only need to show one hypersurface $X$ containing each type of families. For each of the families in Proposition \ref{myfamilies}, we check the conditions in Corollary \ref{inequalityweneed}. We can compute the dimension of the space of global sections using the fact that $h^*\mc{O}_{OG}(d)\otimes \mc{O}_{\Sigma}(-1)=h^*\mc{O}_{OG}(d-1)\otimes \pi^*(\bigwedge^{k-1} S_{k-1}^*))$ and $\pi_*(h^*\mc{O}_{OG}(d)\otimes \mc{O}_{\Sigma}(-1))=Sym^{d-1}((S_{k+1}/S_{k-1})^*)\otimes (\bigwedge^{k-1}S_{k-1}^*)^{d}$. We have:
	\begin{enumerate}
		\item For families in $(1)$, the conditions are:
			$$(k+i) d^2 \le n+1 - 2k -2 -i$$
			for $i=0,\ldots,k-1$.
		\item For family $(2)$:
			$$3kd^2-d^2-d\le n+1-8k-4$$
	\end{enumerate}
	It is easy to see that $(2)$ implies the other conditions. Thus under Assumption \ref{rsshyp}, we have the desired families.\\
	We still need to check condition $(1)$ from definition \ref{twistabledef}. This is done in the following lemma.
\end{proof}
\begin{lem}
	Let $(\pi:\Sigma\rightarrow \pone,h:\Sigma\rightarrow X,\sigma:\pone\rightarrow \Sigma)$ be any of the three families in Proposition \ref{twistablesinx}. Let $\zeta=h\circ \sigma$. We can chose the families such that $\zeta$ is free (i.e. $\zeta^*T_X$ is globally generated).
\end{lem}
\begin{proof} Let $L_i$ be as in Lemma \ref{descoftwistinglb}. In particular, $L_i\cong \mc{O}_{\Sigma}(-1)$. We will consider the case for the twistable families in $(1)$, the proof for family $(2)$ is identical. We have a short exact sequence on $\Sigma$:
	$$0\rightarrow h^*N_{X/OG}\otimes L\otimes \mc{O}(-\sigma(\pone))\rightarrow h^*N_{X/OG}\otimes L \rightarrow h^*N_{X/OG}\otimes L\vert_{\sigma(\pone)}\rightarrow 0 $$
	An easy calculation shows: $h^*N_{X/OG}\otimes L\otimes \mc{O}(-\sigma(\pone))=\mc{O}_{\Sigma}(d-2)\otimes\pi^*\mc{O}(d(k+i-1)+2)$ for the $i$-th family in $(1)$. In particular, we have: $h^1(\Sigma,h^*N_{X/OG}\otimes L\otimes \mc{O}(-\sigma(\pone)))=0$ in all cases. Thus the map $H^0(\Sigma,h^*N_{X/OG}\otimes L)\rightarrow H^0(\pone,\sigma^*(h^*N_{X/OG}\otimes L))$ is surjective. Moreover, by the construction in Corollary \ref{inequalityweneed}, we also have that
	$H^0(\Sigma,h^*T_{OG}\otimes L)\rightarrow H^0(\Sigma,h^*N_{X/OG}\otimes L)$ is surjective. Consequently we get: $H^0(\pone, \zeta^*T_{OG}\otimes L)\rightarrow H^0(\pone,\sigma^*(h^*N_{X/OG}\otimes L))$ is surjective. Since $X$ is non-singular along $\zeta(\pone)$ (Remark \ref{nonsingularalongsigma}), we have a short exact sequence:
	$$0\rightarrow \zeta^*T_X\rightarrow \zeta^*T_{OG}\rightarrow \zeta^*N_{X/OG}\rightarrow 0$$
	Now observe that: $\sigma^*L=\mc{O}(-1)$. Thus twisting the above sequence by $\sigma^*L$ and taking the long exact sequence in cohomology, we get: $H^1(\pone,\zeta^*T_X\otimes \mc{O}(-1))=0$ ($\zeta^*T_{OG}$ is globally generated since all rational curves in $OG$ are free). Therefore, $\zeta^*T_X$ is globally generated. This finishes the proof of Proposition \ref{twistablesinx}.
\end{proof}

\begin{cor}\label{allkindsofcurves}
	There is a twistable curve in $X$ of degree $e$ for all $e \ge k$.
\end{cor}
\begin{proof}
	The curves in family $(1)$ in Proposition \ref{twistablesinx} are of degrees $k,k+1,\ldots,k+(k-1)$. The statement then follows from Lemma \ref{addtwistables}.
\end{proof}


\section{Chains of Rational Lines}\label{chains}

\begin{notn}\label{notation}
	$OG=OG(k,n+1)$, $X\subset OG$ a general degree $d$ hypersurface. $\mbar(X,\tau)$ the space of two pointed $(k+1)$-chains of lines in $X$ as defined in Lemma \ref{stablechainsdefn}. $\mbar_{0,2}(X,k+1)$, the space of two pointed degree $(k+1)$ rational curves in $X$. 
\end{notn}
\begin{lem}\label{chainsinog}
	For a general pair of points $x=(x_1,x_2)\in OG$, let $M^{\tau}_{OG,x}$ denote the fiber of $\mbar(OG,\tau) \rightarrow OG\times OG$ over $x$. Then $M^{\tau}_{OG,x}$ is isomorphic to $Fl(1,2,\ldots,k-1;W)\times H$ where $H$ is a degree $2$, non-singular hypersurface in a projective space of dimension $n-2k$ and $Fl(1,2,\ldots,k-1;W)$ is the full flag variety of chain of subspaces of a $k$-dimensional isotropic subspace $W$.
\end{lem}
\begin{proof}
	For $i=1,2$ let $x_i\in OG=OG(k,V)$ correspond to isotropic subspaces $W_i\subset V$. For this proof generality will mean the following:
	$$W_1^{\perp}\cap W_2 = W_2^{\perp}\cap W_1=\emptyset$$
	Let $U=W_1^{\perp}\cap W_2^{\perp}$; we have $\dim U=n+1-2k$. The generality condition implies:
	$V=W_1\oplus W_2 \oplus U$. The non-degenerate bilinear form on $V$ induces a non-degenerate bilinear form on $U$. Let $\sigma \in ev^{-1}((x_1,x_2))$ be a length $k+1$ chain of lines connecting $x_1$ and $x_2$. $\sigma$ is given by the datum of $(k+2)$ points $x_1=p_0,p_1,\ldots,p_{k+1}=x_2$ such that there is a line connecting each consecutive pair of points. In other words, we have $k$-dimensional isotropic subspaces $W_1=V_0,V_1,\ldots,V_{k+1}=W_2$ such that $V_i\cap V_{i+1}$ is $(k-1)$-dimensional and $V_i+V_{i+1}$ is $(k+1)$-dimensional and isotropic for all $i=0,
	\ldots, k$.\\
	Consider the integer pairs $(a_i,b_i)=(\dim W_1\cap V_i, \dim W_2\cap V_i)$ for $i=0,\ldots k+1$. We know: $(a_0,b_0)=(k,0)$ and $(a_{k+1},b_{k+1})=(0,k)$. It is easy to see that $|a_i-a_{i+1}| \le 1$ and $|b_i-b_{i+1}|\le 1$. Moreover we must have $(a_1,b_1)=(k-1,0)$ since $W_2\cap W_1^{\perp}=\emptyset$ and $(a_{k},b_{k})=(0,k-1)$ since $W_1\cap W_2^{\perp}=\emptyset$. Hence we must have $(a_i,b_i)=(k-i,i-1)$ for $1\le i \le k$. Let $V_{i,1}=W_1\cap V_i$ and $V_{i,2}=W_2\cap V_i$. Observe that, from the generality condition, the induced map $\phi_i:W_2 \rightarrow V_{i,1}^*$ is surjective and hence $\dim(\ker \phi_i)=i$. Since we must have $V_{i+1,2}\subset \ker \phi_i$, $V_{i+1,2}=\ker \phi_i$.\\
	Additionally, consider $U_i=U\cap V_i$. Since $V=W_1\oplus W_2 \oplus U$, $\dim U_i \le 1$ for all $i=1,\ldots,k$. Consider a vector $v^{\prime}\in V_1 \setminus V_{1,1}$. Since $v^{\prime}\in W_1^{\perp}$, $v^{\prime}=w_1+u$ for $w_1\in W_1, u\in U$. Let $w_2\in V_{2,2}$. By the above discussion, $w_2\in V_{1,1}^{\perp}$. But we also have $0=Q(w_2,w_1+u)=Q(w_2,w_1)$. Since $w_2\notin W_1^{\perp}$, this can only happen if $w_1\in V_{1,1}$. Thus, $u\in V_1$ and $\dim U_1=1$. It is easy to see that $U_k=U_{k-1}=\ldots=U_1$ and $V_i=V_{i,1}\oplus V_{i,2}\oplus U_1$ (for $1\le i\le k$). Therefore, $\sigma$ is determined by the $1$-dimensional, isotropic subspace $U_1\subset U$ and the full flag $W_1=V_{0,1}\supset V_{1,1}\supset \ldots \supset V_{k+1,1}=0$. Conversely it is easy to see that this datum defines a chain of lines connecting $x_1$ and $x_2$. Thus $ev^{-1}((x_1,x_2))\cong Fl(1,2,\ldots,k-1,W_1)\times OG(1,U)$. Since $OG(1,U)$ is a degree $2$ hypersurface in $\proj{U}$, this finishes the proof.

\end{proof}
Assume $x=([W_1],[W_2])$ is a general point in $X\times X$. Let $M^{\tau}_{X,x}$ be the fiber of $ev: \mbar(X,\tau)\rightarrow X\times X$ over $x$.
\begin{lem}
	For a general $X$ and a general pair of points on $x\in X\times X$, $M^{\tau}_{X,x}$ is non-singular.
\end{lem}
\begin{proof}
	Let $U\subset OG\times OG$ denote the open subset satisfying the generality condition in the proof of Lemma \ref{chainsinog}. Let $M_U$ denote the inverse image of $U$ in $\mbar(OG,\tau)$. Denote by $\proj{V_d}$ the parameter space of degree $d$ hypersurfaces. Let $\mc{X}\subset U\times \proj{V_d}$ denote the locus where the points are contained in the hypersurface. Let $Z\subset M_U\times_U \mc{X}$ be the locus where the chain of lines is contained in the hypersurface. Let $\pi_1:Z\rightarrow M_U$ and $\pi_2:Z\rightarrow \mc{X}$ denote the projections. Let $(W_1,W_2)$ and $(W_1^{\prime},W_2^{\prime})$ be two pairs of points in $U$. We have:
	$$V=W_1\oplus W_2 \oplus (W_1^{\perp}\cap W_2^{\perp})=W_1^{\prime}\oplus W_2^{\prime} \oplus (W_1^{\prime\perp}\cap W_2^{\prime\perp})$$
	We can chose an isomorphism $\phi_1:W_1\rightarrow W_1^{\prime}$ sending any chain in $Fl(1,\ldots,k-1;W_1)$ to any chain $Fl(1,\ldots,k-1;W_1^{\prime})$. This induces a map $\phi_2:W_2\rightarrow W_2^{\prime}$ via the isomorphisms $W_1^*\cong W_2$ and $W_1^{\prime *}\cong W_2^{\prime}$. Moreover we can chose an isomorphism $\phi_3:W_1^{\perp}\cap W_2^{\perp}\rightarrow W_1^{\prime\perp}\cap W_2^{\prime\perp}$ that preserves the induced non-degenerate bilinear form and send an isotropic element to any element. Setting $\phi=\sum_{i=1}^3 \phi_i:V\rightarrow V$, it is easy to see that $\phi$ is in the orthogonal group. Thus any two chain of lines in $M_U$ can be mapped to each other by the orthogonal group. Therefore, the fibers of $\pi_1$ are all isomorphic (in particular they are all linear subvarieties in $\proj{V_d}$). Therefore $Z$ is non-singular. Hence a general fiber of $\pi_2$ is non-singular.
\end{proof}
Here is a useful lemma for possibly non-complete intersections:
\begin{lem}\label{noncifano}
	Let $Y$ be a Fano variety. Let $(s_i\in L_i)_{1\le i\le c}$ be sections of ample lines bundles on $Y$ such that $K_Y\otimes L_1\otimes\ldots\otimes L_c$ is anti-ample. Let $Z$ be the common zero locus of $(s_i)_{1\le i \le c}$. Assume $Z$ is non-singular. Then $Z$ is Fano as well.
\end{lem}
\begin{proof}
	This is obvious when $\codim(Z,Y)=c$. Assume $\codim(Z,Y)<c$. Let $N_{Z/Y}$ denote the normal bundle of $Z$ in $Y$. There is a surjective map on $Z$: $\oplus_{i=1}^c L_i^*\rightarrow N_{Z/Y}^*$. Let $K$ be the kernel of this map. $K$ is locally free. Taking dual, we get:
	$$0\rightarrow N_{Z/Y}\rightarrow \oplus_{i=1}^c L_i\rightarrow K^*\rightarrow 0$$
	$K^*$ is an ample vector bundle since $\oplus_{i=1}^c L_i$ is ample. Thus $\det K$ is anti-ample. We have: $\det(N_{Z/Y})=\otimes_{i=1}^c L_i \otimes \det K$. Since $K_Z=K_Y\otimes \det(N_{Z/Y})=K_Y\otimes L_1\otimes\ldots\otimes L_c \otimes \det K$, $K_Z$ is anti-ample and $Z$ is Fano.
\end{proof}

\begin{prop}\label{fibratiobyfanos}
	Under Assumption \ref{rsshyp}, there is a dominant map $\phi:M^{\tau}_{X,x}\rightarrow Fl(1,2,\ldots,k-1;W_1)$ whose general fiber is a Fano (specifically rationally connected) variety in $H$. In particular, $M^{\tau}_{X,x}$ is rationally connected.
\end{prop}
\begin{proof}
	Consider $M^{\tau}_{X,x} \subset M^{\tau}_{OG,x}=Fl(1,\ldots,k-1;W_1)\times H$ and look at the projection map $\phi:M^{\tau}_{X,x}\rightarrow Fl(1,2,\ldots,k-1;W_1)$. Fix a flag $F: W_{0,1}\subset W_{1,1}\subset\ldots\subset W_{k-1,1}\subset W_1$. Let $H_F$ denote $\{F\}\times H \subset M^{\tau}_{OG,x}$. Let $(h_i:\Sigma_i\rightarrow OG, \pi_i:\Sigma_{i}\rightarrow H_F,\sigma_{i,0},\sigma_{i,1})_{0\le i\le k}$ denote the universal families of lines over $H_F$. For each $i$, $X$ induces a section of $\pi_{i*}h_i^*\mc{O}_{OG}(d)$. Note that for $i=0,\ldots,k-1$: $\sigma_{i,1}=\sigma_{i+1,0}$. Moreover, proceeding inductively, for $i=0,\ldots,k-1$, we may assume $X$ already contains $h_i\circ \sigma_{i,0}(H_F)$. And for $i=k$: $X$ contains $h_i\circ \sigma_{i,0}(H_F)$ and $h_i\circ \sigma_{i,1}(H_F)$. Thus we need to compute the bundles $(\pi_{i*}h_i^*\mc{O}_{OG}(d)\otimes \mc{O}(-\sigma_{i,0}(H_F)))_{i=0,\ldots,k-1}$ and $(\pi_{k*}h_k^*\mc{O}_{OG}(d)\otimes \mc{O}(-\sigma_{k,0}(H_F)-\sigma_{k,1}(H_F))$. These can be computed similar to Lemma \ref{zerosetofvectorbundle}. Below are the descriptions:
	\begin{enumerate}
		\item For $i=0$: $Sym^{d-1}(\mc{O}_{H_F}\oplus \mc{O}_{H_F}(1))\otimes \mc{O}_{H_F}(1)$.
		\item For $i=1,\ldots,k-1$: $Sym^{d-1}(\mc{O}_{H_F}^{\oplus 2})\otimes \mc{O}_{H_F}(d)$
		\item For $i=k$: $Sym^{d-2}(\mc{O}_{H_F}(1)\otimes \mc{O}_{H_F})\otimes \mc{O}_{H_F}(1)$
	\end{enumerate}
	Thus we get $(k+1)d-1$ many equations whose vanishing locus is the fiber of $\phi$. Since $\dim H_F=n-2k-1$, the fibers of $\phi$ are non-empty and hence a general fiber is non-singular. Moreover the total degree of the equations is $kd^2 < n-2k-1=\deg (-K_{H_F})$. Therefore, from Lemma \ref{noncifano}, a general fiber of $\phi$ is Fano. From \cite[Corollary ~1.3]{ghs}, we have that $M^{\tau}_{X,x}$ is rationally connected.
\end{proof}


\section{Degree $k+1$ Rational Curves}\label{sectiondeformfromchains}

We will use the notation from Notation \ref{notation}. Suppose $\phi:\proj{}^1 \rightarrow \mbar(X,\tau)$ is a family of two pointed chains of rational lines. Furthermore, assume $\phi(\proj{}^1)$ lies in the smooth locus of $ev: \mbar(X,\tau) \rightarrow X\times X$. Let $ev_2:\mbar(X,k+1)\rightarrow X\times X$ denote the evaluation map for two pointed degree $k+1$ stable maps to $X$. Then there is a short exact sequence:
$$0\rightarrow \phi^*T_{ev}\rightarrow (\phi\circ i)^*T_{ev_2}\rightarrow \phi^*\mc{N}_i \rightarrow 0$$

In this section we will construct a $\phi$ such that $(\phi\circ i)^*T_{ev_2}$ is ample. From the above exact sequence, it is enough to construct $\phi$ such that $T_{ev}$ and the normal bundle is ample.

\subsection{Ample Relative Tangent Bundle}

Let $f:\proj{}^1\rightarrow \mbar(X,\tau)$ be a family of chains of lines such that $ev\circ f$ passes through a general point of $X\times X$. Suppose $f$ is given by the datum $(\pi_i:\Sigma_i\rightarrow \proj{}^1, h_i:\Sigma_i\rightarrow X,\sigma_{i,1},\sigma_{i,2})_{0\le i \le k}$ of $k+1$ families of lines in $X$ such that $(\sigma_{i,j})_{j=1,2}$ are non-intersecting sections of $\pi_i$ and $h_i\circ \sigma_{i,2}=h_{i+1}\circ \sigma_{i+1,1}$.    
\begin{lem}\cite[Lemma ~9.2.2]{cox1999mirror} \label{descnormalbundle}With the above notation:
	$f^*\mc{N}_i\cong \oplus_{i=0,k-1} L_i$ where 
	$$L_i\cong \mc{N}_{\sigma_{i,2}}\otimes \mc{N}_{\sigma_{i+1,1}}$$
\end{lem}
Now assume $g:\proj{}^1\rightarrow M^{\tau}_{X,x} \subset Fl(1,2,\ldots,k-1;W_1)\times H$ is a family of chains of rational curves passing through two general fixed points $[W_1]$ and $[W_2]$. Then $g$ is determined by a filtration of vector bundles:
$$0=S_0\subset S_1\subset S_2\subset \ldots \subset S_{k-1}\subset W_1\otimes \mc{O}_{\proj{}^1}$$
And a line sub-bundle: $T\subset U\otimes \mc{O}_{\proj{}^1}$.
We have an easy lemma for later use:
\begin{lem}
	Let $(\pi:\Sigma\rightarrow \pone, \sigma:\pone\rightarrow \Sigma)$ be a family of $1$-pointed lines given by the trio of vector bundles $F_{k-1}\subset F_k\subset F_{k+1}$ on $\pone$ such that $\Sigma=\proj{F_{k+1}/F_{k-1}}$ and $\sigma$ is induced by $F_k/F_{k-1}\subset F_{k+1}/F_{k-1}$. Then the normal bundle of $\sigma(\pone)$: $N_{\sigma}=\mc{O}_{\pone}(\deg F_{k+1}+\deg F_{k-1} - 2\deg F_{k})$.
\end{lem}

\begin{lem}\label{multsofnormalbundle}
	In the above situation and with the notation from Lemma \ref{descnormalbundle}, $L_i\cong \mc{O}(a_i)$ where:
	\begin{enumerate}
		\item $a_0=-\deg T + \deg S_{k-2} - \deg S_{k-1}$
		\item For $1\le i \le k-2$, $a_i=\deg S_{k-i+1}-\deg S_{k-i} - \deg S_{k-i-1}+\deg S_{k-i-2}$
		\item $a_{k-1}=-\deg T +\deg S_2 - \deg S_1$
	\end{enumerate}
\end{lem}
\begin{proof}
	Let $E_i$ denote the kernel to the map $W_2\rightarrow S_{k-i}^*$. Then $\deg E_i=\deg S_{k-i}$.\\
	$i=0$ : By construction in Lemma \ref{chainsinog}, the family $\Sigma_0\rightarrow \pone$ with section $\sigma_{0,2}$ is given by $S_{k-1}\subset S_{k-1}\oplus T\subset W_1\oplus T$ and $\Sigma_1\rightarrow \pone$ with section $\sigma_{1,1}$ is given by $S_{k-2}\oplus T\subset S_{k-1}\oplus T\subset S_{k-1}\oplus T\oplus E_1$. Thus, $\deg \mc{N}_{\sigma_{0,2}}=\deg W_1 + \deg T + \deg S_{k-1}-2\deg S_{k-1} - 2\deg T=-\deg T-\deg S_{k-1}$ and $\deg \mc{N}_{\sigma_{1,1}}=\deg S_{k-1}+\deg T+\deg E_1+\deg S_{k-2}+\deg T-2\deg S_{k-1}-2\deg T=\deg S_{k-2}$. Adding them, we get the result.\\
	$1\le i \le k-2$ : Similarly, ($\Sigma_{i}\rightarrow \pone,\sigma_{i,2}$) is given by $S_{k-i-1}\oplus T\oplus E_{i-1}\subset S_{k-i-1}\oplus T \oplus E_i\subset S_{k-i}\oplus T \oplus E_i$ and ($\Sigma_{i+1}\rightarrow \pone,\sigma_{i+1,1}$) is given by $S_{k-i-2}\oplus T\oplus E_{i}\subset S_{k-i-1}\oplus T \oplus E_i\subset S_{k-i-1}\oplus T \oplus E_{i+1}$. Thus:\\
	$\deg \mc{N}_{\sigma_{i,2}}=\deg S_{k-i}+\deg T+\deg E_i+\deg S_{k-i-1}+\deg T+\deg E_{i-1}-2(\deg S_{k-i-1}+\deg T+\deg E_i)=-\deg S_{k-i-1}+\deg S_{k-i+1}$ and\\
	$\deg \mc{N}_{\sigma_{i+1,1}}=\deg S_{k-i-1}+\deg T+\deg E_{i+1}+\deg S_{k-i-2}+\deg T+\deg E_i-2(\deg S_{k-i-1}+\deg T +\deg E_i)=\deg S_{k-i-2}-\deg S_{k-i}$. Adding them we get the result.\\
	The case for $i=k-1$ is similar to the case $i=0$.
\end{proof}
Chose a sub-bundle $S_{k-1}$ of $W_1\otimes \mc{O}_{\pone}$ given by $S_{k-1} = \oplus_{j=1}^{k-1} \mc{O}(-j)$.
Construct a filtration of $W_1\otimes \mc{O}_{\pone}$ by taking the direct summands of $S_{k-1}: $ $S_i = \oplus_{j=k-i}^{k-1} \mc{O}(-j)$. This induces a map $t:\proj{}^1\rightarrow Fl(1,2,\ldots,k-1;W_1)$. Let $\phi^t:M^{\tau,t}\rightarrow \proj{}^1$ be the base-change via $t$ of the map $\phi$ from Lemma \ref{fibratiobyfanos}.

\begin{lem}\label{primarysection}
	With the above notation, there is a section $s$ of $\phi^t$ such that $s^*T_{\phi}$ is ample and $s^*\mc{N}_i$ is ample.
\end{lem}
\begin{proof}
	Since the general fiber of $\phi^t$ is rationally connected, there is a section $s$ of $\phi_t$ such that $s^*T_{\phi}$ is ample. Moreover, we may assume that the associated sub line-bundle $T\subset U\otimes \mc{O}_{\proj{}^1}$ has arbitrarily high degree. Applying Lemma \ref{multsofnormalbundle}, we see that $a_i = 2$ for all $1\le i \le k-2$ and $a_0, a_{k-1} > 0$. Hence $s^*\mc{N}_i$ is ample.
\end{proof}

\begin{lem}
	Let $F$ denote the flag variety $Fl(1,2,\ldots,k-1;W_1)$. Let $\mc{S}_i$ be the tautological locally free sheaf of rank $i$ on $F$. There is a filtration of the tangent bundle $T_F$, $0=E_k\subset E_{k-1}\subset\ldots E_1\subset T_F$ such that the $i$-th successive quotient is  isomorphic to $W_1/\mc{S}_i\otimes (\mc{S}_i/\mc{S}_{i-1})^*$.
\end{lem}
\begin{proof}\label{flagvarietyfiltration}
	There is a sequence of maps $\phi_i:Fl(1,2,\ldots,k-1;W_1)\rightarrow F_{i}=Fl(1,2,\ldots,i;W_1)$ for $1\le i\le k-1$. Pulling back the respective tangent bundles, we note that there is a short exact sequence:
	$$0\rightarrow W_1/\mc{S}_{i}\otimes (\mc{S}_{i}/\mc{S}_{i-1})^*\rightarrow \phi_i^*T_{F_i}\rightarrow \phi_{i-1}^*T_{F_{i-1}}\rightarrow 0$$
	Let $E_i=Ker(d\phi_i:T_F\rightarrow T_{F_i})$. Then $E_i\subset E_{i-1}$ and from the snake lemma $E_{i-1}/E_{i}\cong W_1/\mc{S}_{i}\otimes (\mc{S}_{i}/\mc{S}_{i-1})^*$. This is our desired filtration.
\end{proof}

\begin{prop}\label{ampletevonfibers}
	Let $ev: \mbar(X,\tau)\rightarrow X\times X$ be the natural evaluation map. There are rational curves $s:\proj{}^1\rightarrow \mbar(X,\tau)$ contained in a general fiber of $ev$, passing through a general point, such that $s^*T_{ev}$ is ample and $s^*\mc{N}_i$ is ample.
\end{prop}
\begin{proof}
	Choose a general pair of points $x=([W_1],[W_2])\in X\times X$ such that $M^{\tau}_{X,x}$ is non-singular. From Lemma \ref{primarysection} let $s:\proj{}^1\rightarrow M^{\tau}_{X,x}$ be a rational curve such that $s^*T_{\phi}$ is ample. Let $F=Fl(1,2,\ldots,k-1)$. Pulling back the filtration of Lemma \ref{flagvarietyfiltration} via $\phi\circ s=t$, $(S_{i}/S_{i-1})^*$ is ample and $W_1/S_i$ is globally generated. Thus $s^*\phi^*T_F$ is ample. There is a short exact sequence on $\proj{}^1$:
	$$0\rightarrow s^*T_{\phi}\rightarrow s^*T_{M^{\tau}_{X,x}}\rightarrow s^*\phi^*T_F\rightarrow 0$$
	Thus $s^*T_{M^{\tau}_{X,x}}\cong s^*T_{ev}$ is ample. Since $s^*T_{M^{\tau}_{X,x}}$ is ample, we may deform $s$ to pass through a general point in $M^{\tau}_{X,x}$.
\end{proof}

\begin{prop}\label{rationallyconnected}
	There is an irreducible component $M_{2,k+1}\subset\mbar_{0,2}(X,k+1)$ such that geometric generic fiber of $ev_2:M_{2,k+1}\rightarrow X\times X$ is rationally connected.
\end{prop}
\begin{proof}
	We chose $M_{0,k+1}$ to be the unique irreducible component of $\mbar_{0,0}(X,k+1)$ defined in Lemma \ref{uniqueirredcomp}. Let $M_{2,k+1}$ be the component of $\mbar_{0,2}(X,k+1)$ dominating this component. In particular, $\mbar(X,\tau)\subset M_{2,k+1}$. Let $ev_{\tau}$ and $ev_2$ be the corresponding evaluation maps. Chose a rational curve  $h:\pone \rightarrow X\times X$ passing through a general point. Let $M_{\tau,h}$ and $M_{k+1,h}$ be the base-change of $h$ via $ev_{\tau}$ and $ev_2$ respectively. Let $M^s_{\tau,h}$ and $M^s_{k+1,h}$ be the corresponding desingularizations. Note that this does not change the generic fibers over $ev_{\tau}$ and $ev_2$. We abuse our notation to continue to use the names $ev_{\tau}$ and $ev_2$ for the maps from the desingularization to $\pone$. Let $g:\pone\rightarrow  M^s_{\tau,h}$ be a section of $ev_{\tau}$. Then $g(\pone)$ passes through the smooth locus of $ev_{\tau}$ and therefore the smooth locus of $ev_{k+1}$. Moreover, we might attach enough rational curves from Lemma \ref{ampletevonfibers} and deform to ensure that $g^*T_{ev_{\tau}}$ is ample and the normal bundle of $M^s_{\tau,h} \subset M^s_{k+1,h}$ is ample (observe that this normal bundle is simply $\mc{N}_i$ for the rational curves lying on a general fiber). Thus, we have that $g^*T_{ev_{k+1}}$ is ample as well. We use the following Lemma to complete the proof.
\end{proof}
\begin{lem}\label{relativelyveryfreefibers}
	Let $\pi: Y\rightarrow \proj{}^1$ be a morphism of algebraic spaces. Assume that there exists a section $s: \proj{}^1 \rightarrow Y$ lying in the smooth locus of $\pi$, such that $s^*T_{\pi}$ is an ample vector bundle. Then: a general fiber of $\pi$ is rationally connected.
\end{lem}
\begin{proof}
	Consider the space $Mor(\pone,Y)$, parametrizing morphisms from $\pone$ to $Y$. Denote by $T_Y$ the dual to $\Omega_Y$. There is a short exact sequence on $\proj{}^1$:
	$$0\rightarrow s^*T_{\pi}\rightarrow s^*T_Y\rightarrow T_{\pone}\rightarrow 0$$
	Thus $s^*T_Y$ is ample and $[s]$ is an unobstructed point in $Mor(\proj{}^1,Y)$. Let $M$ denote the irreducible component of $Mor(\proj{}^1,Y)$ containing $[s]$. By considering the degree of $s^{*}\pi^{*}\mc{O}(1)$, we see that every point of $M$ parameterizes a section of $\pi$. Since the evaluation map $M\times \proj{}^1\rightarrow Y$ is smooth at all points $([s],t)$ for $t\in \proj{}^1$, we conclude that there is a section $s^{\prime}$ passing through a general point of $Y$ such that $s^{\prime*}T_{\pi}$ is ample. Thus we may assume $s(0)=p$ for a general point $p\in Y$. Consider the space of morphisms $Mor(\proj{}^1,Y)|_{0\mapsto p}$ from $\proj{}^1$ to $Y$ with $0\in \proj{}^1$ mapping to $p\in Y$. There is a morphism $Mor(\pi):Mor(\proj{}^1,Y)|_{0\mapsto p}\rightarrow Mor(\pone,\pone)_{0\mapsto 0}$. Let $M_1$ and $M_2$ denote the irreducible components of  $Mor(\proj{}^1,Y)|_{0\mapsto p}$ and $Mor(\pone,\pone)_{0\mapsto 0}$ containing $[s]$ and $[id]$ respectively. The tangent spaces of $M_1$ and $M_2$ at $[s]$ and $[id]$ are $H^0(\pone,s^*T_Y(-1))$ and $H^0(\pone,T_{\pone}(-1))$ respectively. Twisting the above short exact sequence by $\mc{O}_{\pone}(-1)$ and taking the long exact sequence in cohomology, we see that $Mor(\pi)$ is smooth at $[s]$ and has fiber dimension $h^0(\pone,s^*T_{\pi}(-1))\ge 1$. Thus, by an application of the rigidity lemma (\cite[p. ~43]{abelianvarieties}), $s$ must specialize to a reducible curve $\tilde{s}:C\rightarrow Y$ with at least two irreducible components passing through $p$. Again, by considering the degree of $\tilde{s}^{*}\pi^{*}\mc{O}(1)$, $C$ has exactly one component $C_0$ such that $\tilde{s}|_{C_0}$ is a section of $\pi$ and the other components are contracted by $\pi$. Thus there is a rational curve through $p$ contained in the fiber over $\pi$. Since $p$ was a general point, a general fiber of $\pi$ is uniruled.\\
	Now consider the relative MRC quotient (see \cite{relmrc} for a construction): $g:Z\rightarrow \proj{}^1$, an algebraic space over $\proj{}^1$ and $f:Y\dashrightarrow Z$ a rational map such that $g\circ f=\pi$. Let $U\subset Y$ be the locus where $f$ is defined. A general fiber of $f$ is birationally rationally connected and any rational curve meeting a general fiber of $f$ is contracted by $f$. Therefore, by \cite[Theorem 1.1]{ghs} the fibers of $g$ are not uniruled. Now, we may assume, $s(\proj{}^1)$ meets $U$. Thus $f\circ s$ is a section of $g$ and hence lies in the smooth locus of $g$. Moreover $s(\proj{}^1)\cap U$ also lies in the locus where $f$ is smooth. Thus pulling back the relative tangent bundles to $\proj{}^1$, there is a map:
	$$s^*T_{\pi}\rightarrow (f\circ s)^* T_g$$
	This is a surjection over $s^{-1}(s(\proj{}^1)\cap U)$ and hence the co-kernel is a torsion sheaf. Since $s^*T_{\pi}$ is ample, this can only happen if $(f\circ s)^* T_g$ is ample as well. Thus the argument above proving uniruledness of fibers of $\pi$ also applies to $g$ as long as $Z$ is at least $2$-dimensional. This is a contradiction, hence we must have that $g$ is an isomorphism. Therefore a general fiber of $\pi$ is rationally connected.
\end{proof}

\section{Proof of Main Theorem}\label{sectionmaintheorem}

Proof of the main theorem depends on the following lemma that allows us to induct from degree $e$ rational curves to degree $e+1$ rational curves.
\begin{lem}\cite[Lemma 8.6]{nk1006}\label{inductionstatement} Let $M_{2,e}$ be the irreducible component defined in Lemma \ref{uniqueirredcomp} for $e\ge 1$. Assume the space of lines through a general point in $X$ is irreducible and rationally connected. Assume that there is a twistable curve in $X$ of degree $e$ such that the corresponding family $\Sigma$ is isomorphic to $\pone\times\pone$. Furthermore, assume that the map:
	$$ev_{e}: M_{2,e}\rightarrow X\times X$$
	has rationally connected geometric generic fiber. Then the map $$ev_{e+1}: M_{2,e+1}\rightarrow X\times X$$ has rationally connected geometric generic fiber as well.
\end{lem}
\begin{proof}[Proof of Theorem \ref{mainthm}]
	The first condition from Definition \ref{rss} follows from Theorem \ref{irredlines}. The existence of a very twistable curve in $X\subset OG$ now follows from Proposition \ref{twistablesinx}. Proposition \ref{rationallyconnected} provides the base case for rationally connectedness of fibers of the evaluation maps. Now, from Corollary \ref{allkindsofcurves} and Lemma \ref{inductionstatement}, we conclude that $ev:M_{2,e}\rightarrow X\times X$ has rationally connected geometric generic fiber for all $e\ge k+1$. This concludes the proof of the theorem.
\end{proof}
\bibliography{refs}{}
\bibliographystyle{amsalpha}
\end{document}